\documentclass[10pt,reqno,final]{amsart}
\usepackage{amsfonts}
\usepackage[notcite,notref]{showkeys}
\usepackage{url,hyperref,multirow}
\usepackage{color}
\usepackage{cases}
\usepackage{subfigure}
\usepackage{comment}

\usepackage{epsfig,amsmath,bm,caption2,epstopdf}
\usepackage{amssymb,version,graphicx,fancybox,mathrsfs,pifont,booktabs,wrapfig,mathtools}

\catcode`\@=11 \theoremstyle{plain}
\@addtoreset{equation}{section}   

\@addtoreset{figure}{section}
\renewcommand\thefigure{\thesection.\@arabic\c@figure}
\@addtoreset{table}{section}
\renewcommand\thetable{\thesection.\@arabic\c@table}

\newtheorem{thm}{\bf Theorem}[section]
\newtheorem{cor}{\bf Corollary}[section]

\newenvironment{corollary}{\begin{cor}} {\end{cor}}
\newtheorem{lmm}{\bf Lemma}[section]

\newenvironment{lemma}{\begin{lmm}}{\end{lmm}}
\theoremstyle{remark}
\newtheorem{remark}{\bf Remark}[section]

\def \ri {{\rm i}}

\newcommand{\bs}[1]{\boldsymbol{#1}}

\usepackage{graphicx}

\begin{document}
	\graphicspath{{./Figures/}}

	\title[RCL with real transformation] {A Novel PML-type Technique for Acoustic Scattering Problems based on A Real Coordinate Transformation}
	 \author[J. Wang, \,   L. Wang \,  \& \, B. Wang]{Jiangxing Wang${}^{1}$, \;  Li-Lian Wang${}^{2}$ \; and \; Bo Wang${}^{1}$}	

	\thanks{${}^{1}$MOE-LCSM,
	 School of Mathematics and Statistics, Hunan Normal University, Changsha, Hunan, 410081, China and Xiangjiang Laboratory, Changsha, 410205, China. The research of the authors is partially supported by the Major Program of Xiangjiang Laboratory(No.22XJ01013), NSFC (grant No. 52331002, 12022104, 12371394), Key Project of Hunan Provincial Department of Education (grant No. 22A033) and the Changsha municipal natural science fundation (No. kq2208158). Emails: jxwang@hunnu.edu.cn (J. Wang); bowang@hunnu.edu.cn (B. Wang).\\
	\indent ${}^{2}$Division of Mathematical Sciences, School of Physical and Mathematical Sciences, Nanyang Technological University, 637371, Singapore.  Email: lilian@ntu.edu.sg (L.-L. Wang).
 }

	\begin{abstract} It is known that any {\em real coordinate transformation} (RCT)  to compress waves in an unbounded domain into a bounded domain results in infinite oscillations   that cannot be resolved by any grid-based method.  In this paper, we intend to show that it is viable if  the outgoing waves are compressed along the radial direction and the resulting oscillatory pattern is extracted explicitly.  We therefore construct a perfectly matched layer (PML)-type technique for domain reduction of wave  scattering problems using RCT, termed as  real compressed layer (RCL). Different from all existing approaches, the RCL technique has two features: (i) the RCL-equation only involves real-valued coefficients, which is more desirable for computation and analysis; and (ii) the layer is not ``artificial'' in the sense that the computed field in the layer can recover the outgoing wave of the original scattering problem in the unbounded domain. Here we demonstrate the essential idea and performance of the RCL for the two-dimensional Helmholtz problem with  a bounded scatterer, but this technique can be extended to three dimensions in a similar  setting.
 	\end{abstract}
	\maketitle
	
	\section{Introduction}
	We propose a new  PML-type technique
 to reduce the  time-harmonic wave scattering problem governed by the exterior Helmholtz equation
	\begin{subequations} \label{helmholtz2eq}
		\begin{align}
			&\mathcal{H}[U]:=-\Delta  U(\hat{\bs x})-k^2 U({\hat{\bs x}})=0,\quad \hat{\bs x}\in\Omega^e:=\mathbb R^2\backslash D,\label{helmholtz2eq2}\\
			& U(\hat{\bs x})=g(\hat{\bs x}), \quad\hat{\bs x}\in\Gamma_D=\partial{D},\label{helmholtz2bc1}\\
			&\frac{\partial  U} {\partial \rho} -\ri k U= o(\rho^{-1/2}) \;\;{\rm as}\;\; \rho = |{\hat{\bs x}}| \to \infty, \label{helmholtz2bc2}
		\end{align}
	\end{subequations}
to a bounded domain, where $k>0$ denotes the wave number, $D\subset {\mathbb R}^2$ is a bounded scatterer with Lipschitz boundary $\Gamma_{\!D}$ and
$g$ is a given incident wave.

Similar to the widely-used PML approach  introduced by Berenger \cite{berenger1996three,berenger1994perfectly}, we enclose the region of interest by a layer but construct the equation therein  very differently.
In contrast to many  existing techniques  based on the complex coordinate stretching/transformation, we use a real compression coordinate transformation.
However, as commented in \cite{johnson2008notes} ``{\it any real coordinate mapping from an infinite to a finite domain will result in solutions that oscillate infinitely fast as the boundary is approached -- such fast oscillations cannot be represented by any finite-resolution grid, and will instead effectively form a reflecting hard wall.}" Much of this paper is to show that the real compression coordinate transformation indeed works when it is properly integrated with another technique.

Before we elaborate on this technique, we feel compelled to briefly review the relevant existing methods to  motivate this new technique and demonstrate its distinction. Due to the fundamental importance in applications, the Helmholtz scattering problems have been extensively studied in the literature. For the well-posedness, we refer  to \cite{colton2013integral}. For numerical computations, one viable  approach  
	 is to introduce a boundary element method (BEM) based on an integral representation 
 \cite{colton2013integral}. However, the BEM has  limited capability to deal with complex scatterers and/or inhomogeneous media. Another  approach  is to truncate the domain and impose the artificial boundary conditions based on the  Dirichlet-to-Neumann (DtN) technique.
  The DtN notion is only available for special geometries and  involve global series representations,  so it is typically complicated to implement  \cite{keller1989exact,wang2012fast,yang2016seamless,zhang2019seamless}.
   To overcome the drawbacks of the DtN boundaries, some localized variants were introduced in \cite{bayliss1980radiation,engquist1977absorbing} which
    are local and hence easy to implement, but they are low order and not fully non-reflecting at times.
	
	A commonly-used  approach  for the unbounded domain reduction 
	is the PML which was first introduced by Berenger \cite{berenger1996three,berenger1994perfectly} in time domain. The basic idea of the PML method is to surround the computational domain by a layer filled with specially designed lossy media aiming to attenuate all waves scattered from  the interior region. The governing Helmholtz equation is then modified in such a way so that any outgoing wave is perfectly transmitted from the domain into the layer and then damped, regardless of the incident angle. In the frequency domain, the idea of constructing PML  can be simply interpreted as a complex coordinate stretch in the governing equations \cite{chew19943d}. Since then, various constructions of PML techniques
	have been proposed and well studied in the literature \cite{chew19943d,turkel1998absorbing}. Moreover, they have been populated into major software package such as the COMSOL Multiphysics. Very recently, a truly exact perfect absorbing layer (PAL) with general star-shaped domain truncation of the exterior Helmholtz equation was introduced in \cite{wang2017perfect,yang2019truly}. It is remarkable to point out that the truncation by the PAL is truly exact in the sense that the PAL solution is identical to the original solution in the inner domain \cite{yang2019truly}.
	It is noteworthy that the transformation in the PAL approach is quite different from the transformation in the PML approach \cite{bermudez2007optimal,chen2005adaptive,chew19943d}. Indeed, the mapping used in PAL method can be viewed as a composite mapping with a composition of a real mapping that compresses $\rho\in(a,\infty)$ into $r\in(a,a+d)$ and the complex mapping used in the PML approach with a differential absorbing function. However, the real compression mapping might cause  oscillations of the PAL solution in the layer and result in degenerating coefficients in the PAL equation. In order to extract the essential oscillation and remove the singular coefficients, the authors introduced a judicious substitution by separating the PAL solution as a product of an oscillatory  part, a singular factor and a well-behaved part, so one can use the finite element/spectral element methods to resolve this well-behaved part.
	
	The convergence of the PML method has been drawn many researchers' attention in the past years \cite{bao2010adaptive,bao2023convergence,hohage2003solving,lassas1998existence,lassas2001analysis,zhang2018high,zhang2022exponential,zhang2023higher}. It is proved by Lassas and Smoersalo in \cite{lassas1998existence} that the PML solution convergent exponentially to the Helmholtz scattering problem for the circular and smooth PML layers as the thickness of the layer tends to infinity. The anisotropic PML method in which the PML layer is placed outside of a rectangle or cubic domain. The exponentially convergence property of the anisotroptic PML is proved in \cite{chen2013anisotropic,liang2016convergence}. Further,the convergence of the uniaxial PML method has been considered recently by Chen and Liu \cite{chen2005adaptive}, Bramble and Pasciak \cite{bramble2010analysis}, Chen and Zheng \cite{chen2010convergence}. In practical application, the adaptive PML method whose main idea is to use the posteriori error estimate to determine the PML parameter and use the adaptive finite element method to solve the PML equation has been studied in \cite{chen2008posteriori,chen2013adaptive,chen2005adaptive,chen2008adaptive}.

Observe that in the polar coordinates $(\rho, \theta),$ the far-field outgoing wave of the exterior Helmholtz problem \eqref{helmholtz2eq} has a well separation of decay and oscillation (see  the related analysis in Subsection \ref{section:idea} below):
\begin{equation} \label{solu00}
		U(\rho,\theta)=\sqrt{\frac{2}{\pi k\rho}}e^{{\rm i}k\rho}M(k\rho,\theta),
	\end{equation}
where $M$ is  a well-behaved function. The PML and PAL techniques enforce the field $U(\rho,\theta)$ to decay exponentially in the layer through the complex coordinate transformation: $\rho=\rho_{_R}(r)+\ri \rho_{_I}(r),$ so it results in the exponential decaying factor from the complex exponential:
$e^{{\rm i}k\rho}= e^{-k\rho_{_I}(r)} e^{{\rm i}k\rho_{_R}(r)}$ in the transformed layer with coordinates $(r,\theta)$. The PAL further diminishes
the oscillations caused by the real part of the transformation  and also uses a substitution.

We design a new layer from a different perspective using a real-valued  mapping $\rho=\rho(r)$ that can dramatically accelerate
the slow decaying factor $1/\sqrt{\rho}= 1/\sqrt{\rho(r)}$ where $\rho(r)$ is an exponential mapping so that the transformed field decays exponentially (see \eqref{ascir} below). Then we introduce a substitution to diminish the oscillation of the RCL-solution in the layer like the PAL technique in \cite{wang2017perfect,yang2019truly}. It is anticipated that the real transformation leads to the RCL-equation of real coefficients which is advantageous  for implementation.  On the other hand, the equation in the layer is not artificial and the computed field can provide an approximation to the  far-field by applying the inverse transformation, which appears
 an important advantage for the use of a real mapping. However, this is not possible for the PML and PAL methods based on the complex transformations.

	The outline of this paper is as follows. In Section \ref{sect2A}, we demonstrate the essential idea for the circular RCL and prove the RCL-solution convergent exponentially to the Helmholtz scattering problem in the reduced domain. In Section \ref{Sect3:RRCL}, we introduce the RCL method with a rectangular layer truncation for the Helmholtz scattering problem and conduct the convergence analysis. We also provide ample numerical results to show the good performance of the proposal technique.

	\section{Proof of concept  via circular RCL}\label{sect2A}
	In this section, we demonstrate the essential idea through the construction of circular RCL, as this is a relatively simpler setting for clarity of exposition and ease of comparisons.
	
	\subsection{Essence of circular RCL}\label{section:idea}
	Let $B_a = \{\hat{\bs x}\in \mathbb{R}^2:|\hat{\bs x}|<a\}$ be a suitable disk that contains the scatterer $D$ and  the support of the source term (see Figure \ref{circlearea} (left), where $\Gamma_a=\partial B_a$). It is known that the solution of the  Helmholtz problem \eqref{helmholtz2eq} exterior to $B_a$ can  be expressed  in the polar coordinates - $(\rho,\theta)$ (see Karp \cite[Theorem 1]{karp1961convergent} and Villamizar et al.
	\cite{Villamizar20}):
	\begin{equation}\label{karpsolution}
		U( \rho,\theta) =  H_0^{(1)}(k \rho)\sum_{l=0}^\infty\frac{F_l(\theta)}{(k\rho)^l} + H_1^{(1)}(k \rho)\sum_{l=0}^\infty\frac{G_l(\theta)}{(k\rho)^l},
	\end{equation}
	which converges uniformly and absolutely for  $\rho>a$ and $\theta\in [0, 2\pi)$. Here the coefficients $\{F_l, G_l\}$ in $\theta$ can be determined recursively by the boundary data at $r=a$ (see   \cite[(6)-(7)]{Villamizar20}). The two  Hankel functions in this Karp's expansion   have the following
representations (see \cite[p. 229]{olver2010nist}):
	\begin{equation*}
		\begin{split}
			&H_0^{(1)}(z) = \Big(\frac{2}{\pi z}\Big)^{1/2}e^{{\rm i}(z-\pi/4)}\Big(1-\frac{1}{8z}{\rm i}+\frac{3}{128z^2}{\rm i}^2+\cdots\Big),\\ 
			&H_1^{(1)}(z) = \Big(\frac{2}{\pi z}\Big)^{1/2}e^{{\rm i}(z-3\pi/4)}\Big(1+\frac{3}{8z}{\rm i}+\frac{3}{128z^2}{\rm i}^2+\cdots\Big), 
		\end{split}
	\end{equation*}
	for $ -\pi+\delta \leq {\rm ph} z\leq 2\pi-\delta$ with some small $\delta>0.$

	\begin{figure}[!htb]
		\centering
		\subfigure[RCL domain]{
			\includegraphics[width=0.23\textwidth]{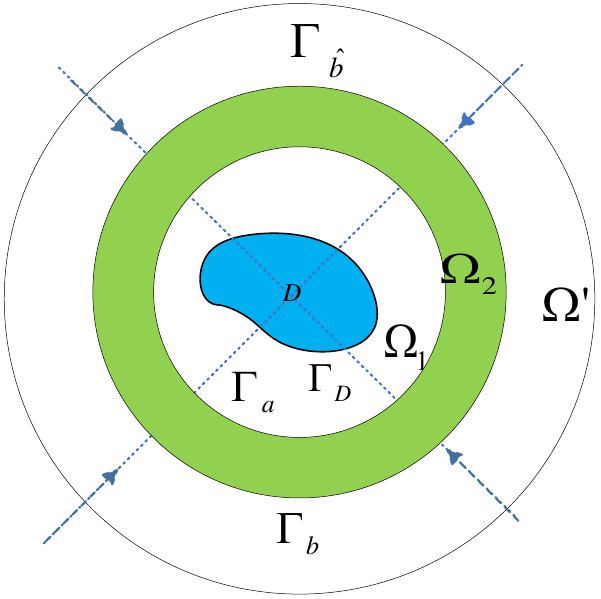}
		} \qquad\qquad
		\subfigure[PML domain]{
			\includegraphics[width=0.23\textwidth]{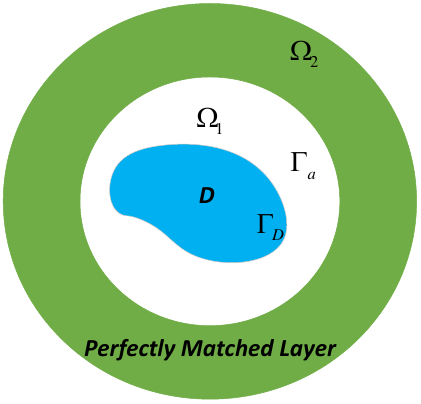}
		}
		\caption{Schematic illustration of the circular  RCL and PML domains.}\label{circlearea} 
	\end{figure}

We rewrite the solution \eqref{karpsolution} as
	\begin{equation} \label{solu}
		U(\rho,\theta)=\sqrt{\frac{2}{\pi k\rho}}e^{{\rm i}k\rho}M(k\rho,\theta),
	\end{equation}
	where
	\begin{equation*}
		\begin{split}
			M&(\rho,\theta) = e^{-\frac{\pi}{4}{\rm i}}\Big\{F_0(\theta)+\frac{1}{\rho}\Big(F_1(\theta)-\frac{\ri}{8}F_0(\theta)\Big)+\frac{1}{\rho^2}\Big(F_2(\theta)-\frac{\ri}{8}F_1(\theta)
			+\frac{3\ri}{128}F_0(\theta)\Big)+\cdots\Big\}\\
			&+e^{-\frac{3\pi}{4}{\rm i}}\Big\{G_0(\theta)+\frac{1}{\rho}\Big(G_1(\theta)+\frac{3\ri}{8}G_0(\theta)\Big)+\frac{1}{\rho^2}\Big(G_2(\theta)+\frac{3\ri}{8}G_1(\theta)+\frac{3\ri}{128}G_0(\theta)\Big)+\cdots\Big\}.
		\end{split}
	\end{equation*}

\medskip
\noindent\underline{\bf Observation}\,
\medskip

In  polar coordinates,   the outgoing wave  \eqref{solu} has a clear separation of {\em decay} and {\em oscillation}. More precisely,
	\begin{itemize}
		\item[(a)] it decays slowly at the rate: $1/\sqrt{k\rho}$\,;
\smallskip
		\item[(b)] it exhibits the oscillatory pattern: $e^{{\rm i}k\rho}$\,;
\smallskip
		\item[(c)] the function $M(k\rho,\theta)$ essentially  has no oscillation, as $k$ only appears in the inverse powers.
	\end{itemize}
	%

\medskip
\noindent\underline{\bf Conceptual comparison: PML versus RCL}\,
\medskip

To recap, the well-known PML technique \cite{berenger1994perfectly} employs a complex coordinate stretching to attenuate all the waves that propagate
from inside of $B_a.$  As with \cite{collino1998perfectly,chen2005adaptive}, let $\alpha(r)=1+\ri \sigma(r),$ and introduce the complex coordinate transform:
\begin{equation}\label{PML-transform}
\rho=\rho(r)= \begin{dcases}
r & \text { if } r \leq a, \\ \int_0^r \alpha(t)\, d t=r \beta(r) & \text { if } r > a,
\end{dcases}
\end{equation}
where the ``absorbing function'' (ABF)
$\sigma(r)\ge 0$ is a continuous function for $r>0$ and  $\sigma(r)\equiv 0$ for $r<a.$ Thus we can write the transformation as
\begin{equation}\label{complrxtransform}
		\rho = \rho^{\rm R}(r)+{\rm i}\rho^{\rm I}(r)=r+\ri  \int_0^r \sigma(t)\, d t.
	\end{equation}
The typical choices of ABF include the polynomials  \cite{collino1998perfectly,chen2005adaptive} and rational functions \cite{bermudez2007optimal}.
Such a transformation enforces the oscillatory factor (b) for $r>a$ in \eqref{solu} delays exponentially as
\begin{equation}\label{decayexp}
		e^{{\rm i}k {\rho}} = e^{{\rm i}k \rho^{\rm R}(r)}e^{-\rho^{\rm I}(r)}={\mathcal O}(e^{-\rho^{\rm I}(r)}).
	\end{equation}
The PAL technique recently proposed by \cite{wang2017perfect,yang2019truly} used  singular rational mappings for both
$\rho^{\rm R}(r)$ and $\rho^{\rm I}(r)$ in \eqref{complrxtransform}.


\smallskip

In distinct contrast with PML, we introduce a {\em real exponential transform}  to render the slow decaying factor (a): $1/\sqrt{k\rho}$ in \eqref{solu} decay exponentially fast in the new coordinates. More precisely, we adopt
	\begin{equation}\label{rrmap}
		\rho = \tau(r):=
		\begin{cases}
			r & \text { if } r\le a, \\
			ae^{\tau_0(r-a)} & \text { if } r> a,
		\end{cases}
	\end{equation}
	but remain the angular variable $\theta\in [0,2\pi)$ unchanged.
	In \eqref{rrmap},  $\tau_0>0$ is a tuning parameter. Formally, the field  \eqref{solu} for $r>a$  is transformed into
	\begin{equation}\label{ascir}
	\begin{split}
		 u(r,\theta) & := U(\rho,\theta)=\sqrt{\frac{2}{k\pi a}}\, e^{-\frac{\tau_0}{2}(r-a)}\,{\rm exp}\big({\rm i}k a\, e^{\tau_0(r-a)}\big)\, M(k\tau(r),\theta)\\
		&=\mathcal O(e^{-\frac{\tau_0}{2}(r-a)}),
		\end{split}
	\end{equation}
	which decays exponentially in $r,$ and
	$$M(k\tau(r),\theta)\sim \big\{F_0(\theta)-G_0(\theta)+\ri (F_0(\theta)+G_0(\theta))\big\}/\sqrt{2},\quad r\gg 1.$$
However,   when $a\, e^{\tau_0(r-a)}>r,$  the oscillation in the neighbourhood  $r\in (a,a+\delta)$ (for some $\delta>0$) may increase.
Nevertheless,  the oscillatory factor can be explicitly extracted as follows
	\begin{equation}\label{asciruv}
		 u(r,\theta) = e^{\ri k\tau(r)}  v(r,\theta)=  e^{\tau_0(r-a)}  v(r,\theta),\quad r>a, \;\; \theta\in [0,2\pi),
	\end{equation}
	where $ v(r,\theta)$ decays exponentially without essential oscillations.
		\begin{figure}[!t]
	\centering
\subfigure[{${\rm Re}(U(\rho,0))$ with $\rho\in [2,24]$}]{
	\includegraphics[width=0.40\textwidth]{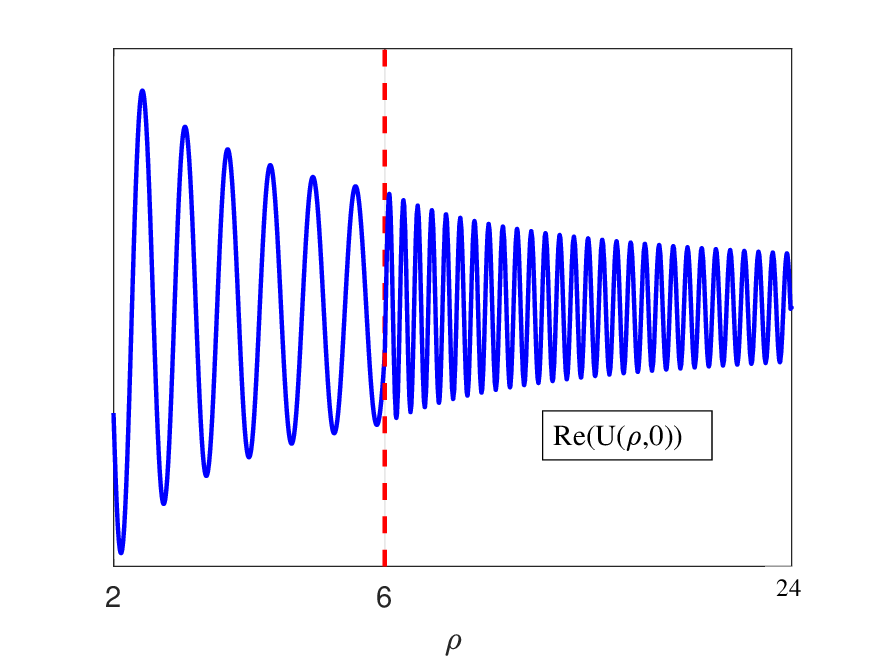}
}\qquad
\subfigure[{${\rm Re}(U(\rho,\pi/4))$ with $\rho\in [2,24]$}]{
	\includegraphics[width=0.40\textwidth]{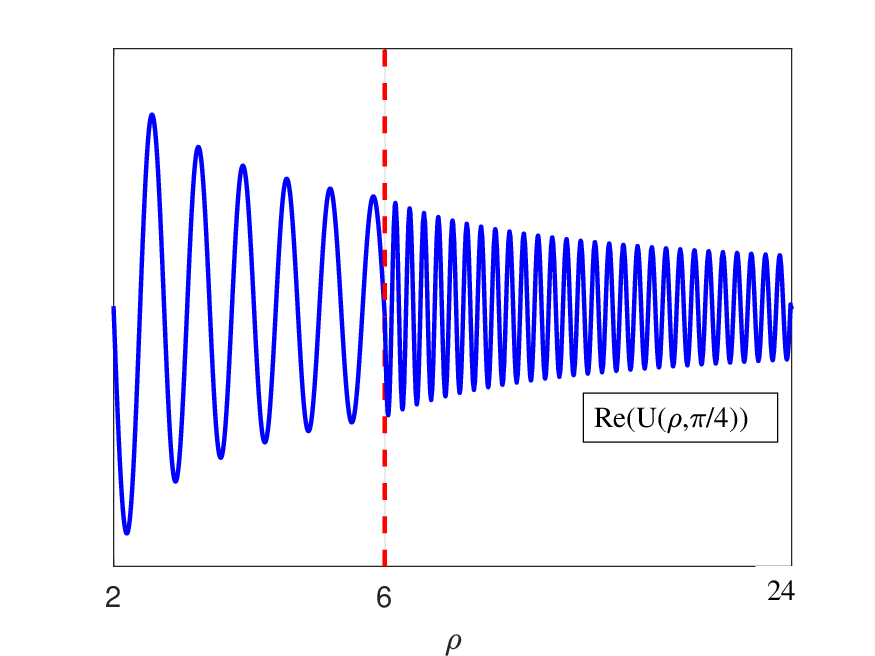}
}\\
	\subfigure[{${\rm Re}\{u(r,0),v(r,0)\}$ with $r\in [2,12]$}]{
		\includegraphics[width=0.40\textwidth]{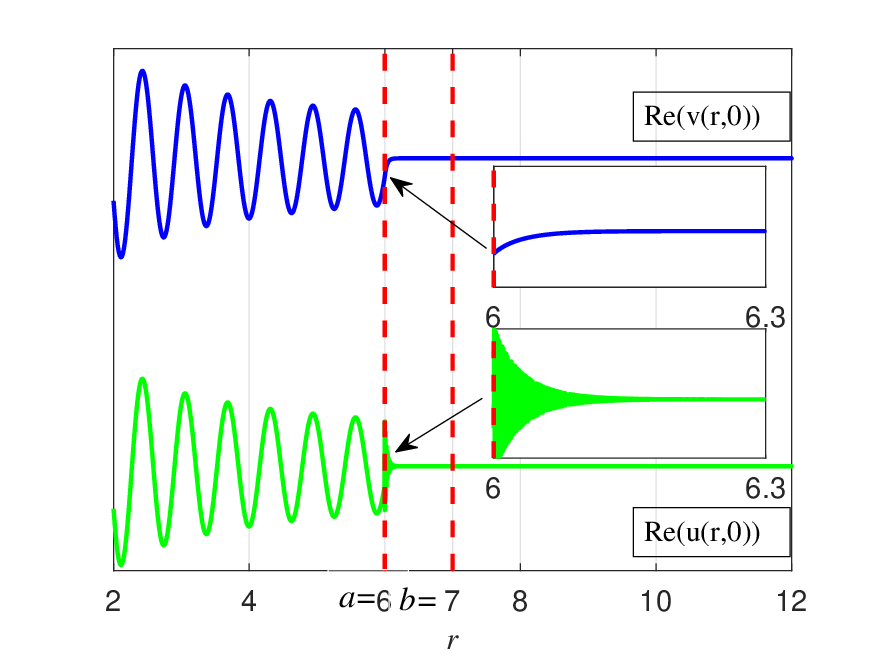}
	}\qquad
	\subfigure[{${\rm Re}\{u(r,\pi/4),v(r,\pi/4)\}$ with $r\in [2,12]$}]{
		\includegraphics[width=0.40\textwidth]{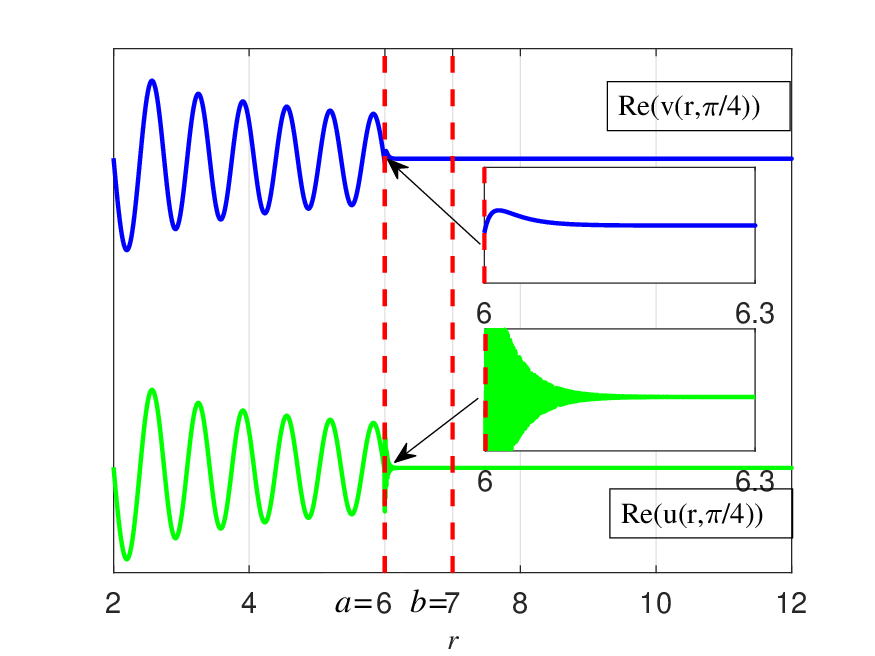}
	}
	\caption{Profiles of the real part of $U(\rho,\theta), u(r,\theta), v(r,\theta)$ with $\theta=0,\pi/4$.}\label{figexact1drpml}
\end{figure}

	As an  illustrative example, we consider the exterior Helmholtz problem \eqref{helmholtz2eq}  with a circular scatterer $D$ with radius $R>0$
and	the planar  incident wave $ g (\theta)= e^{{\rm i}k R\cos \theta},$ which admits the series solution (cf. \cite{yang2019truly}):
	\begin{equation}\label{solutioncir}
		U(\rho,\theta) = -\sum_{|n|=0}^{\infty}\frac{{\rm i}^{n}J_n(k R)}{H_n^{(1)}(kR)}H_n^{(1)}(k\rho)e^{{\rm i} n\theta}.
	\end{equation}
In Figure \ref{figexact1drpml} (a)-(b), we plot the profiles of $U(\rho,\theta)$ with $R=2, a=6, \tau_0=60, k=10$ and $\theta=0,\pi/4,$ where we use  different coordinate scalings for $\rho\in [R,a]$ and $\rho>a$ to show the oscillation for the comparison with
compare with the profiles of  ${u}(r,\theta)$ and ${v}(r,\theta)$ for $r>a$  in
Figure \ref{figexact1drpml} (c)-(d). Note that  $U(\rho,\theta)=u(r,\theta)=v(r,\theta)$ in the annulus $\Omega_1=\{R<r<a\}$ (where $\rho=r$),
but they behave very differently when $r>a.$ The profiles clearly show that $U(\rho,\theta)$ decays  slowly in $\rho,$
but $u(r,\theta), v(r,\theta)$ decays exponentially fast in $r.$  Due to the compression, $u(r,\theta)$ has big oscillations near $r=a,$
  but thanks to \eqref{asciruv},  $v(r,\theta)$ essentially has no oscillation.

	\subsection{The RCL-equation}
   Using the coordinate  transformation \eqref{rrmap}, we can convert the original Helmholtz equation \eqref{helmholtz2eq}  (exterior to the disk $B_a$ in  polar coordinates $(\rho,\theta)$) into the following problem
	in the new coordinates  ${\bs x}=(r\cos\theta,r\sin\theta)$:  
	\begin{align}\label{RPMLeq}
		-\mathbb{J}^{-1}\nabla\cdot\left({\bs C}\,\nabla {u}(\bs x)\right)-k^2 {u}(\bs x) = 0 \quad\hbox{in}\;\;\; \Omega^e=\mathbb R^2\setminus  \bar B_a,
	\end{align}
	supplemented with the same Dirichlet boundary condition on $\Gamma_{D}$ and  far-field condition transformed from \eqref{helmholtz2bc2}.
Here, the coefficient matrix ${\bs C} = {\bs J}^{-1}({\bs J}^{-1})^\top{\mathbb J},$ where $\bs J$ is the Jacobian matrix, $\mathbb J={\rm det}(\bs J)$ is the Jacobian and they can be computed from  \eqref{rrmap} readily as
	\begin{equation}\label{invjacobianr}
	\begin{split}
		& {\bs J}^{-1}=\frac{1}{\tau'} {\bs R}_{\theta}\begin{pmatrix}
			1 & 0\\
			0 & r\tau'/\tau
		\end{pmatrix} {\bs R}_{\theta}^{\top},\quad
			{\bs R}_{\theta}=\begin{pmatrix}
			\cos\theta & \sin\theta \\
			-\sin\theta & \cos\theta
		\end{pmatrix},\\
	&\mathbb{J}^{-1}=\frac 1 {{\rm det}(\bs J)}=\frac{r}{\tau\tau'},\quad \tau'=\frac{\rm d\tau} {{\rm d} r}=\tau'(r).
	\end{split}
	\end{equation}

 In view of \eqref{ascir},   {\em the solution  $u(r,\theta)$ of the transformed Helmholtz problem \eqref{RPMLeq} must  decay exponentially}  {\rm(}due to the factor $e^{-\frac{\tau_0}{2}(r-a)}$ in \eqref{ascir}{\rm).} This motivates us to truncate \eqref{RPMLeq}  directly
by a disk $B_b$ with a suitable radius $b>a,$  and impose the  homogeneous Dirichlet boundary condition at $r=b,$ which leads to the {\em  Helmholtz equation reduced by  the RCL technique} or simply the  {\bf RCL-equation}: 
	\begin{subequations} \label{RCL}
		\begin{align}
			& -\nabla\cdot\left({\bs C}\nabla \hat{u}(\bs x)\right)-k^2 n(\bs x) \hat{u}(\bs x)= 0\quad\hbox{in}\;\; \Omega:= B_b\setminus \bar{D}, \label{RCLeq1}\\
			& \hat{u}(\bs x) = g\quad\hbox{on} \;\; \partial D; \quad  \hat{u}(\bs x) =0\quad\hbox{on}\;\; \Gamma_b:=\partial B_b,\label{BC2}
		\end{align}
	\end{subequations}
	together with the usual transmission conditions at $r=a.$
As shown in Figure \ref{circlearea} (left), the computational domain consists of $\Omega=\Omega_1\cup \Omega_2$, where $\Omega_1$ is the domain of interest in simulating the scattering wave and  $\Omega_2=\{a<| {\bs x}|<b\}$ is the circular  RCL layer.  Note that
in $\Omega_2,$   $n=\mathbb J$ and $\bs C$ are given in \eqref{RPMLeq}-\eqref{invjacobianr}, while in $\Omega_1,$ $\bs C, n$ are inherited from the original Helmholtz equation in $\Omega^e.$ 


\begin{remark}\label{Rmk:realcoef}
{\em Different from the PML techniques based upon complex coordinate transformations, the variable coefficients here are all real-valued in the layer.  Moreover,  the field in the RCL is not fictitious that can provide a good approximation to the original field exterior to $\Gamma_a.$}
\end{remark}

	\subsection{Convergence analysis}
	We show that the $H^1$-error  between
	the solution $ \hat u({\bs x})$  of the boundary value problem \eqref{RCL}
	and the compressed scattering field $u({\bs x})$ of \eqref{RPMLeq} is exponentially small. This implies the non-reflectiveness of the truncation  and is essential for the success of this new  technique. 
	
	As some preparations, we first  derive the following uniform bounds for the ratios of  Hankel functions.
	\begin{lemma}\label{Hankel}
		For $\rho>a>0$ and $ka>1$, we have
		\begin{equation}
			\;\;\left|\frac{H_0^{(1)}(k \rho)}{H_0^{(1)}(k a)}\right|\leq 4\sqrt{\frac{a}{\rho}};  \;\quad 		
			\left|\frac{H_n^{(1)}(k \rho)}{H_n^{(1)}(k a)}\right|\leq \sqrt{\frac{a}{\rho}},\quad n= \pm 1, \pm 2,\cdots. \label{hanke1kn0}
		\end{equation}
	\end{lemma}
	\begin{proof} In view of the identity $H_{-n}^{(1)}(z) = e^{{\rm i} n\pi}H_n^{(1)}(z)$, we only need to prove \eqref{hanke1kn0} for
		positive integer $n.$  According to \cite[p. 446]{watson1995treatise}, the function  $x|H_n^{(1)}(x)|^2$ with $n\ge 1$ is a strictly decreasing on $ (0,\infty)$, so we have
		\begin{equation}
			\bigg|\frac{H_n^{(1)}(k \rho)}{H_n^{(1)}(k a)}\bigg|^2=\frac{k \rho|H_n^{(1)}(k \rho)|^2}{k a|H_n^{(1)}(k a)|^2}\frac{k a}{k\rho}\leq \frac{a}{\rho},\quad n\geq 1,\label{hanke1kn}
		\end{equation}
		which leads to the second bound for $n\geq 1$.
		
		For  $n=0$, we rewrite the ratio and use the estimate \eqref{hanke1kn} with $n=1$ to obtain
		\begin{eqnarray}
			\left|\frac{H_0^{(1)}(k \rho)}{H_0^{(1)}(k a)}\right|=\left|\frac{H_0^{(1)}(k \rho)}{H_1^{(1)}(k \rho)}\right|\left|\frac{H_1^{(1)}(k \rho)}{H_1^{(1)}(k a)}\right|\left|\frac{H_1^{(1)}(k a)}{H_0^{(1)}(k a)}\right|\leq \left|\frac{H_0^{(1)}(k \rho)}{H_1^{(1)}(k \rho)}\right|\left|\frac{H_1^{(1)}(k a)}{H_0^{(1)}(k a)}\right|\sqrt{\frac{a}{\rho}}.\label{hankelrationest}
		\end{eqnarray}
		Recall the property (cf. \cite{watson1995treatise,chen2005adaptive}):
		$$|H_{n-1}^{(1)}(x)|\leq |H_n^{(1)}(x)|,\quad x>0,\quad n\geq 1,$$
		so we have
		\begin{equation}
			\left|\frac{H_{0}(k\rho)}{H_1(k\rho)}\right|\leq 1.\label{hankelk021}
		\end{equation}
		For the second ratio, we employ the expansion of the Hankel function (cf. \cite[10.17.13]{olver2010nist}):
		\begin{equation}
			\begin{split}
				&H_0^{(1)}(x) = \Big(\frac{2}{\pi x}\Big)^{1/2}e^{{\rm i}(x-\pi/4)}(1+R_1(0,x)),\nonumber\\
				&H_1^{(1)}(x) = \Big(\frac{2}{\pi x}\Big)^{1/2}e^{{\rm i}(x-3\pi/4)}(1+R_1(1,x)),\nonumber
			\end{split}
		\end{equation}
		and the estimates (cf. \cite[10.17.14]{olver2010nist}):
		\begin{equation*}
			|R_1(n,x)|\leq 2|a_1(n)|x^{-1}e^{|n^2-1/4|x^{-1}},\quad a_1(0)=-\frac{1}{8},\quad a_1(1)=\frac{3}{8},\label{lemma12}
		\end{equation*}
		to obtain
		\begin{equation}
			\left|\frac{H_1^{(1)}(x)}{H_0^{(1)}(x)}\right|\leq \frac{1+2|a_1(1)|{x^{-1}}e^{\frac{3}{4x}}}{1-2|a_1(0)|{x^{-1}}e^{\frac{1}{4x}}}=\frac{4x+3e^{\frac{3}{4x}}}{4x-e^{\frac{1}{4x}}}\leq 4, \quad \forall x>1. \label{hankel01est}
		\end{equation}
		In view of \eqref{hankelk021}-\eqref{hankel01est}, we obtain the desired bound for $n=0$ from
		\eqref{hankelrationest} immediately.
	\end{proof}

	For $f\in L^2(\Gamma_c)$  on the circle $\Gamma_c=\{|{\bs x}|=c\}$, we define the Sobolev (trace) norm
	\begin{equation}
		\|f\|_{H^{s}(\Gamma_c)}^2 = \sum_{|n|=0}^{\infty}(1+n^2)^{s}|\hat{f}_n|^2,\quad \hat{f}_n = \frac{1}{2\pi}\int_0^{2\pi}f(c\cos\theta,c\sin\theta) e^{-{\rm i} n\theta}d\theta,
		\label{negnorm1}
	\end{equation}
	for $s\ge 0.$

	We have the following error bounds, which provide us a practical rule to choose $\tau_0$ and the width of the layer to ensure $e^{-\frac{1}{2}\tau_0(b-a)}\ll 1.$
	\begin{thm}\label{cirpcl}
		The RCL-equation \eqref{RCL} has a unique solution $\hat{u}(\bs x)$. Moreover, the error  between  $\hat{u}(\bs x)$ and the mapped solution $ u(\bs x)$ of
		\eqref{RPMLeq} in the $H^1$-norm decays exponentially as follows
		\begin{equation}
			\|u-\hat{u}\|_{H^1(\Omega)}  \leq C\max\{\tau_0 b,(a\tau_0)^{-1}\}e^{-\frac{1}{2}\tau_0(b-a)}\|U\|_{H^{1/2}(\Gamma_a)},\label{cirest1}
		\end{equation}
		where the parameter $\tau_0>1$ and  $C$ is a positive constant independent of $u$.
	\end{thm}
	\begin{proof}
 We first transform 
 the RCL-equation \eqref{RCL} back to the $(\hat{x}, \hat{y})$-coordinates
 by \eqref{rrmap} as 
 \begin{equation} \label{RPMLlayerorgin}
			-{\Delta}\widehat{U}-k^2\widehat{U} = 0\quad\hbox{in}\quad\Omega'; 
			\quad\widehat{U} = g\quad\hbox{on}\quad\partial D; \quad \widehat{U} =0\quad\hbox{on}\quad\Gamma_{\hat {b}},
		\end{equation}
 where we denoted  $\widehat{U}(\rho,\theta): = \hat{u}(\tau^{-1}(\rho),\theta)$  and
		$$ \hat{b} = \tau(b), \quad \Gamma_{\hat {b}} = \big\{{\bs \hat{\bs x}}\in \mathbb R^2:|{\bs {\hat x}}|=\hat{b}\big\},\quad  \Omega'=\tau(\Omega).$$
 Suppose that $k^2$ is not a Dirichlet eigenvalue of $-\Delta$.
It is known that  the problem  \eqref{RPMLlayerorgin}  has a unique weak solution in $ H^1(\Omega')$  (cf. \cite{griesmaier2011error,mclean2000strongly}) satisfying
		\begin{equation}
			\|\widehat{U}\|_{H^1(\Omega')}\leq \widehat C\|g\|_{H^{1/2}(\partial D)},\label{regur}
		\end{equation}
		where the positive constant $\widehat C$ only  depends on $k$ and $|\Omega'|$.	This implies the  existence  and uniqueness of the solution to the RCL-equation \eqref{RCL}.
		
		Subtracting \eqref{RPMLlayerorgin} from \eqref{helmholtz2eq}  leads to  the error equation
		\begin{equation} \label{errorRPMLlayerorgin}
			-{\Delta}e-k^2e = 0\quad\hbox{in}\quad\Omega';
			\quad e = 0\quad\hbox{on}\quad\partial D; \quad e = U\quad\hbox{on}\quad\Gamma_{\hat {b}},
		\end{equation}
where  $e = U({\hat{\bs x}}) - \widehat{U}(\hat{\bs x}).$ Then we infer from \eqref{regur} that
	\begin{align}
		\|e\|_{H^1(\Omega')}\leq \widehat C\|U\|_{H^{1/2}(\Gamma_{\hat{b}})}.\label{estviabdrydata} 
	\end{align}
On the other hand,  the Helmholtz equation \eqref{helmholtz2eq} exterior to $B_a$ admits the series solution  in polar coordinates (cf. \cite{nedelec2013acoustic}):
	\begin{equation}\label{seriesSolu}
		U(\rho,  \theta)=\sum_{|m|=0}^\infty A_m \frac{H_m^{(1)}(k\rho)}{H_m^{(1)}(k a)} e^{\ri m \theta},\quad
A_m: = \frac{1}{2\pi}\int_0^{2\pi}U(a,\theta)e^{-{\rm i}m\theta}d\theta,
	\end{equation}
	for given Dirichlet data $U|_{\rho=a},$
	where  $H_m^{(1)}(\cdot)$ is the Hankel function of the first kind of order $m$.
		Using \eqref{negnorm1} and Lemma \ref{Hankel}, we derive from \eqref{seriesSolu} that
		\begin{equation}\label{qest}
			\begin{split}
				\|U|_{\Gamma_{\hat{b}}}\|_{H^{1/2}(\Gamma_{\hat{b}})}^2 & = 2\pi\sum_{m= -\infty}^\infty(1+m^2)^{1/2}\, \bigg|\frac{H_m^{(1)}(k \tau(b))}{H_m^{(1)}(k a)}\bigg|^2\, |{A}_m|^2 \\
				&\leq 2\pi\sum_{m= -\infty}^\infty(1+m^2)^{1/2}\bigg|\frac{H_m^{(1)}(k \tau(b))}{H_m^{(1)}(k a)}\bigg|^2|{A}_m|^2 \\
				&\leq  \frac{32\pi a}{\tau(b)}\|U|_{\Gamma_{{a}}}\|_{H^{1/2}(\Gamma_a)}^2.
			\end{split}
		\end{equation}
		As $U(a,\theta)=u(a,\theta)$,   we obtain from \eqref{ascir} and \eqref{estviabdrydata}-\eqref{qest} that
		\begin{equation}
			\|e\|_{H^1(\Omega')}\leq 4\pi\sqrt{\frac{2 a}{\tau(b)}}\|U\|_{H^{1/2}(\Gamma_a)}
			= 4\sqrt{2}\pi e^{-\frac{1}{2}\tau_0(b-a)}\|U\|_{H^{1/2}(\Gamma_a)}.\label{erorPML1}
		\end{equation}
		As the mapping $\tau(r)$ is the identity mapping for all $\hat{\bs x}\in\Omega_1=\{\hat{\bs x}\in {\mathbb R}^2:|\hat{\bs x}|\in[R,a]\}$, we have
		\begin{align}
			\|u-\hat{u}\|_{H^1(\Omega_1)} =\|U-\widehat{U}\|_{H^1(\Omega_1)} \leq \|e\|_{H^1(\Omega')} \leq C\widehat C e^{-\frac{1}{2}\tau_0(b-a)}\|U\|_{H^{1/2}(\Gamma_a)}.\label{addest1}
		\end{align}
		On the other hand, from the relations between the coordinates $(x, y)$ and $(\hat x, \hat y)$, we derive
		\begin{equation*}
			\begin{split}
				\|e\|_{H^1(\Omega'\setminus\Omega_1)}^2&=\int_0^{2\pi}\int_a^{\hat{b}}\left[\Big(\frac{\partial e}{\partial \rho}\Big)^2+\Big(\frac{1}{\rho}\frac{\partial e}{\partial\theta}\Big)^2+e^2\right]\rho d\rho d\theta\nonumber\\
				&=\int_0^{2\pi}\int_a^b\left[\Big(\frac{\partial(u-\hat{u})}{\partial r}\frac{\partial r}{\partial \rho}\Big)^2+\Big(\frac{1}{\tau(r)}\frac{\partial(u-\hat{u})}{\partial\theta}\Big)^2+(u-\hat{u})^2\right]\tau_0 \tau^2(r)drd\theta\nonumber\\
				&=\int_0^{2\pi}\int_a^b\left[\frac{1}{\tau_0 r}\Big(\frac{\partial(u-\hat{u})}{\partial r}\Big)^2+\tau_0 r \Big(\frac{1}{r}\frac{\partial(u-\hat{u})}{\partial\theta}\Big)^2 +\frac{\tau_0 \tau^2(r)}{r}(u-\hat{u})^2\right]r drd\theta\nonumber\\
				&\geq\min\{\tau_0^{-1} b^{-1},\tau_0 a\}\| \nabla(u-\hat{u})\|_{L^2(\Omega'\setminus\Omega_1)}^2+\tau_0a\|u-\hat{u}\|_{L^2(\Omega'\setminus\Omega_1)}^2,\nonumber
			\end{split}
		\end{equation*}
		where we use the fact
		$$\frac{\partial r}{\partial\rho}=\frac{1}{\tau_0\tau(r)},\quad \tau(r)\geq r.$$
		Together with \eqref{erorPML1}, we arrive at
		\begin{equation}\label{H12estimate}
			\| u-\hat{u}\|_{H^1(\Omega\setminus\Omega_1)}\leq C\max\{\tau_0 b,\tau_0^{-1} a^{-1}\}e^{-\frac{1}{2}\tau_0(b-a)}\|U\|_{H^{1/2}(\Gamma_a)}.
		\end{equation}
		A combination of  \eqref{addest1} and  \eqref{H12estimate}  completes the proof.
	\end{proof}
	
	\begin{remark}\label{Rmk:22}
		{\em Some remarks are in order.
	\begin{itemize}
		\item[(i)] From the above proof, we see that  the dependence of the constant $C$ on $k$  is inherited from $\widehat C$ of  the standard regularity result \eqref{regur}.
  \smallskip
		\item[(ii)] The estimate \eqref{cirest1} provides us insights into the choices of the  RCL parameter $\tau_0$ and $b.$
		For fixed $k$  and given accuracy threshold $\epsilon>0,$ we can choose $\tau_0(b-a)=O(|\ln \epsilon|).$
	\end{itemize}}
\end{remark}

	\subsection{Performance of the circular RCL}
 We provide below some numerical results to demonstrate the accuracy and good performance of this new technique. As we are interested in the propagation of the outgoing scattering field in the outer layer,
 it appears sufficient to consider the  circular scatterer with $D:=\{0<r< R\}.$ In this case formulate the transformation \eqref{rrmap} as
	\begin{equation}
		\rho = \tau(r) = \int_0^r\alpha(s)ds = r\beta(r)=\begin{cases}
			r & \text { if } r\le a, \\
			ae^{\tau_0(r-a)} & \text { if } r> a,
		\end{cases}\quad \theta\in[0,2\pi), \label{rmap}
	\end{equation}
	where  
	\begin{align*}
		\alpha(r) = 
		\begin{cases}
			1,&  r\in[R,a], \\
			a\tau_0 e^{\tau_0(r-a)},&  r\in (a,+\infty),
		\end{cases}
		\quad\beta(r) =
		\begin{cases}
			1, & r\in[R,a], \\
			\frac{\tau(r)}{r},  &  r\in (a,+\infty).
		\end{cases}
	\end{align*}
Correspondingly,  we have
	\begin{align*}
	&\Omega_{1} = \{{\bs x}\in {\mathbb R}^2:|{\bs x}|\in I_1:= (R,a)\},\quad\Omega_{2} = \{{\bs x}\in {\mathbb R}^2:|{\bs x}|\in I_2:=(a,b)\},\\
	&\Gamma_R =  \{{\bs x}\in {\mathbb R}^2:|{\bs x}|= R\},\quad \Gamma_b= \{{\bs x}\in {\mathbb R}^d:|{\bs x}|= b\},\quad I=I_1\cup I_2.
	\end{align*}

We  employ the Fourier spectral method in the angular direction  and Lengedre spectral-element method in the radial direction  to solve the RCL equation \eqref{RCLeq1}--\eqref{BC2}. Firstly, we  reduce the two-dimensional problem \eqref{RCLeq1}--\eqref{BC2} to a sequence of one-dimensional problems by using the Fourier expansion in $\theta$ direction as in \cite{shen2011spectral}:
	\begin{subequations}\label{RPML00}
		\begin{align}
			&-\frac{1}{r\beta}\frac{\partial }{\partial r}\left(\frac{\beta r}{\alpha}\frac{\partial \hat{u}_n}{\partial r}\right)+\left(\frac{\alpha n^2}{r^2\beta^2}-\alpha k^2\right) \hat{u}_n = 0,\label{RPMLeq1}\\[3pt]
			& \hat{u}_n = \hat{g}_n\; \hbox{ on }\; \Gamma_R,\quad \hat{u}_n =0\; \hbox{ on }\; \Gamma_b, \ |n|=0,1,2,\cdots\label{RPMLbc2}
		\end{align}
	\end{subequations}
		together with usual transmission conditions at $r=a,$
where  $\{\hat u_n, \hat{g}_n\}$ are the Fourier coefficients of $\{\hat u, g\}.$ We reiterate that (i) the coefficients of the RCL-equation \eqref{RPML00} are all real valued;  (ii) its solution is expected to decay exponentially in the layer in view of  \eqref{ascir} and Theorem \ref{cirpcl}; and (iii)   the compression mapping increases oscillation near $r=a$ (see Figure \ref{figexact1drpml}) but with a pattern $e^{{\rm i}k(\tau(r)-a)}$ (see  \eqref{ascir}).  Thus
we introduce  the substitution: $\hat{u}_n(r) = w(r) \hat{v}_n(r)$ with
 \begin{equation}\label{substituionB}
	w(r):=
	\begin{cases}
			1, & r\in I_1, \\
			e^{{\rm i}k(\tau(r)-a)},  &  r\in I_2,
		\end{cases}
 \end{equation}
 and solve for $\hat{v}_n(r)$, which must  be  free of oscillation and well-behaved on $I_2$.  Correspondingly, a weak formulation of \eqref{RPML00} is to find  $\hat{u}_n = w\hat{v}_n$ with $\hat{v}_n\in H^1(I)$, $\hat{v}_n = \hat{g}_n$ on $\partial D$ and $\hat{v}_n=0$ on $\Gamma_b$ such that
	\begin{align}
		\breve{\mathcal{B}}(\hat{v}_n,\phi)=\mathcal{B}(w\hat{v}_n,w\phi)=0, \quad\forall \phi\in H_0^1(I),\label{rbilineardef}
	\end{align}
	where
	the bilinear form $\mathcal{B}(\cdot,\cdot): H^1(I)\times H^1(I)\rightarrow \mathcal{C}$ is defined as
	\begin{equation*}
		\mathcal{B}(\hat{u}_n,\psi) = \int_R^b\Big\{\frac{\beta r}{\alpha}\frac{\partial \hat{u}_n}{\partial r}\frac{\partial}{\partial r}\Big(\frac{1}{r\beta}{\psi^*}\Big)+\Big(\frac{\alpha n^2}{r^2\beta^2}-\alpha k^2\Big)\hat{u}_n{\psi^*} \Big\} dr\label{bilineardef}
	\end{equation*}
	with $\psi^*$ being the conjugate of $\psi.$
	Let  $P_N$ be the polynomial set of degree at most $N$. Define the finite dimensional approximation space
\begin{align*}
	V_N=\big\{\phi\in C(\bar I)\,:\, \phi|_{I_1},\phi|_{I_2}\in P_{N}\big\}.
	\end{align*}
	Then, the spectral element approximation for the RCL problem \eqref{RPMLeq1}--\eqref{RPMLbc2} is to  find $\hat{u}_N^n = w\hat{v}_N^n$ with $\hat{v}_N^n\in V_N,\hat{v}_N^n=\hat{g}_n \hbox{ on }\Gamma_R \hbox{ and } \hat{v}_N^n = 0\hbox{ on } \Gamma_b$ such that 
	\begin{align}
		\breve{\mathcal{B}}(\hat{v}_N^n,\phi)=\mathcal{B}(\omega\hat{v}_N^n,\omega\phi)=0,\quad \forall \phi \in V_N^0=V_N\cap H_0^1(I).
\label{computscheme1}
	\end{align}

In the following test, we consider the incident wave of the form
	$$u_{\rm inc}(\theta)=e^{{\rm i}k R\cos\theta}=\sum\limits_{n=-\infty}^{\infty}{\rm i}^{-n}J_n(kR)e^{{\rm i}n\theta},$$
	and the scattering problem has the exact solution
	\begin{align}
		U_{\rm ex}(\rho,\theta) = -\sum_{n=0}^{\infty}{\rm i}^{n}J_n(k R)\frac{H_n^{(1)}(k \rho)}{H_n^{(1)}(k a)}e^{{\rm i}n\theta}.\label{solution} 
	\end{align}
 Correspondingly, we have the exact mapped solution and reference solution with substitution as follows 	\begin{equation}
		\begin{split}
			{u}(r,\theta) = -\sum_{n=0}^{\infty}{\rm i}^{n}J_n(k R)\frac{H_n^{(1)}(k \tau(r))}{H_n^{(1)}(k a)}e^{{\rm i}n\theta};\quad
		    v(r,\theta) =
			\begin{cases}
				 u(r,\theta), & r\in[R,a), \\
              e^{{\rm i}k(\tau(r)-a)} u(r), & r\geq a.
			\end{cases}\nonumber
		\end{split}
	\end{equation}
	
In the following numerical tests, we choose $\tau_0 = \log(1/(\epsilon^2))/(b-a)$ and the cut-off Fourier mode $M$ is chosen such that $|J_M(k R)|\leq\epsilon_1$ where $\epsilon,\epsilon_1>0$ are given error thresholds. Moreover, let the numerical solution be   
$\hat{v}_N=\sum_{|n|=0}^M\hat{v}_N^n e^{{\rm i}n\theta},$ and denote
$v^R={\rm Re}\{v\}, v^I={\rm Im}\{v\}.$ 
In the first test, we choose $R = 0.5,a=1, b = 2,\epsilon = \epsilon_1 = 10^{-12}$ and using the same order $N$ in the sub-intervals: $I_1$ and $I_2$. We tabulate in Table \ref{table1} the point-wise maximum norm $e_v^{R}= v^{R}-\hat{v}_N^{R}, e_v^{I}= v^{I}-\hat{v}_N^{I}, e_u^{R}= u^{R}-\hat{u}_N^{R}, e_u^{I}= u^{I}-\hat{u}_N^{I}$ which is 
computed at $20000$ uniform points on each interval, from which we observe a typical spectral convergence even for large wave numbers.

\begin{table}[!th]\small
	\caption{The convergence rate of Fourier spectral method with $\theta  = 0, b-a = 1.$} \label{table1}
		\begin{tabular}{|c|c|c|c|c|c|}
			\hline
			$k$&$N$ &$\|e_u^{R}\|_{L^\infty(R,a)}$  &$\|e_u^{I}\|_{L^\infty(R,a)}$   &$\|e_v^{R}\|_{L^\infty(a,b)}$  &$\|e_v^{I}\|_{L^\infty(a,b)}$  \\
			\hline
			50
			&50 & 6.8236{\rm e}-5 &6.1701{\rm e}-5 &5.7900{\rm e}-5 & 4.5228{\rm e}-5 \\
			&80 & 1.4142{\rm e}-8 &1.3521{\rm e}-8 &4.1121{\rm e}-8 & 3.6170{\rm e}-8\\
			&100 &7.6288{\rm e}-11 &7.1791{\rm e}-11 &2.3315{\rm e}-10 & 2.2650{\rm e}-10\\
			\hline
			150
			&80 & 4.5498{\rm e}-3 &4.8001{\rm e}-3 &2.8559{\rm e}-3 & 3.0749{\rm e}-3 \\
			&100 & 3.9752{\rm e}-5 &3.9730{\rm e}-5 &1.3509{\rm e}-5 & 1.6136{\rm e}-5\\
			&120 &1.0533{\rm e}-7 &9.4557{\rm e}-8 &1.7454{\rm e}-7 & 1.7719{\rm e}-7\\
			&150 &8.9893{\rm e}-11 &9.0540{\rm e}-11 &6.2593{\rm e}-10& 6.2760{\rm e}-10\\
			\hline
			300
			&150 & 9.6473{\rm e}-4 &1.0137{\rm e}-3 &3.6435{\rm e}-4 & 3.5747{\rm e}-4\\
			&180 &1.3120{\rm e}-6 &1.1770{\rm e}-6 &1.8190{\rm e}-7 & 1.8122{\rm e}-7\\
			&200 &5.4953{\rm e}-9 &5.4157{\rm e}-9 &7.2197{\rm e}-8 & 9.8913{\rm e}-8\\
			\hline
		\end{tabular}
	\end{table}

We plot in Figure \ref{fufig01}: (a)-(b) the exact solution and numerical RCL solution $k=50$ and $N_1=N_2=N=100$ with $\theta=0$. 
Further, the field in RCL area is well behaved and decrease well to zero without any oscillation. In other words, due to the fact that $\hat{v}_N$ is a good approximation of $v$, $u(r)=e^{-{\rm i}k(\tau(r)-a)}v(r)(r\in I_2)$ and $U(\rho)=u(\tau(r))$, we know that $U_N=e^{-{\rm i}k(\rho-a)}\hat{v}_N(\tau^{-1}(r))$ should be a good approximation of $u$.  We plot the solution of both the original scattering problem and the numerical RCL solution on the interval $\rho\in[1,3]$ for $k=50$ and $N_1=N_2=N=100$ in Figure \ref{fufig01}: (c)-(d).  From this figure, we conclude that the RCL solution convergence well to the original solution in the domain $[a,\tau^{-1}(b)]$. In the other words, the far field pattern of the original scattering problem can be obtain by using coordinate transform directly, other than using the Green function method. In other words, we obtain the far field pattern of the scattering problem by using $U_N=e^{-{\rm i}k(\rho-a)}\hat{v}_N(\tau^{-1}(r))$.

		\begin{figure}[!tb]
		\centering
		\subfigure[Numerical and exact real part of $v$.]{
			\includegraphics[width=0.40\textwidth]{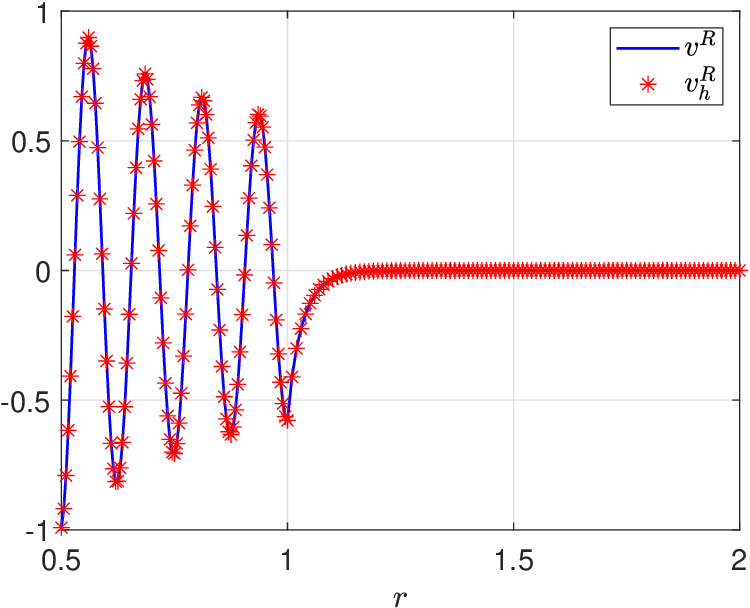}
		}\qquad
		\subfigure[Numerical and exact imaginary part of $v$.]{
			\includegraphics[width=0.40\textwidth]{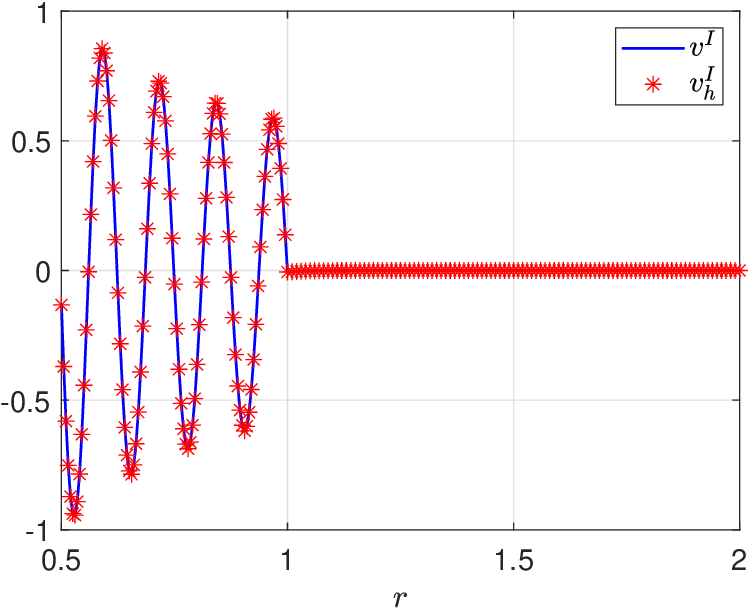}}
   \subfigure[Numerical and exact real part of $u$.]{
			\includegraphics[width=0.40\textwidth]{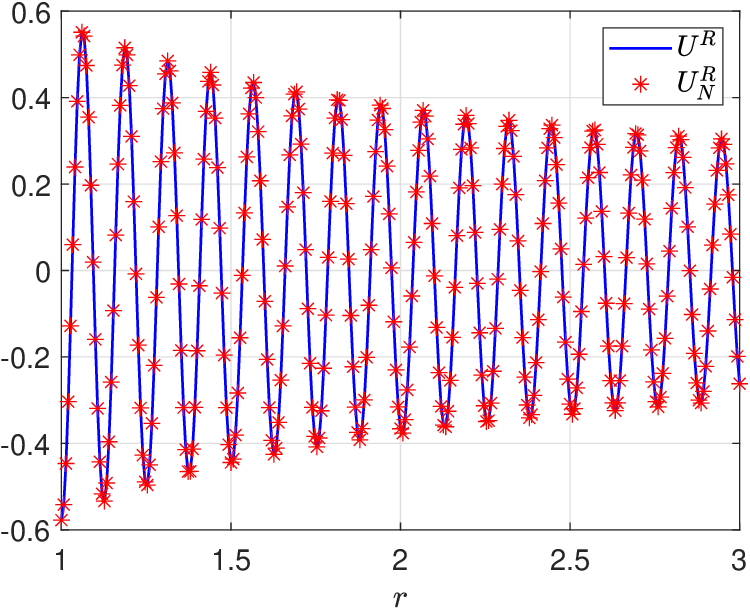}
		}\qquad
		\subfigure[Numerical and exact imaginary part of $u$.]{
			\includegraphics[width=0.40\textwidth]{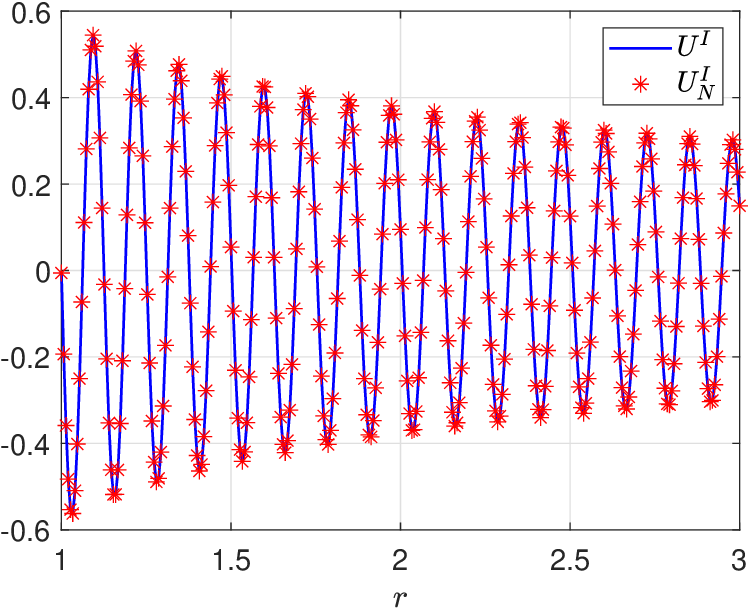}
		}
		\caption{The numerical solution in 2D for $\theta = 0, k  = 50$, under the real mapping and polynomial order is $N_1=N_2=150.$
  }\label{fufig01}
	\end{figure}

	When using the PML/PAL method, one would wish the ``artificial domain" $I_2$ is as thin as possible so that one can save the computational cost. Due to the fact that the RCL method has the effect of PML layer, we give here some numerical example on the impact of the thickness of the layer. To this end, we choose $R = 0.5,a=1,\epsilon = \epsilon_1 = 10^{-13}$ and using the same order of polynomial in $I_1$ and $I_2$. We list the maximum errors on each sub-interval in Table \ref{table01}. 
	 From the two tables, we know that the RCL solution converges  to the original solution in  $I_1$ and $I_2$. The results in Table \ref{table01} indicate that if  $e^{-\tau_0 d}<\epsilon$, then the thickness  $d$ does not affect the computational results essentially. 
	\begin{table}
		\small
		\caption{Errors vs thickness $d=b-a$ for RCL with $(N_1, N_2) = (200,200)$.}
		\label{table01}
		\begin{tabular}{|c|cc|cc|cc|}
			\hline\multirow{2}{*}{$k$} & \multicolumn{2}{|c|}{$d=1$} & \multicolumn{2}{c|}{$d=0.1$} & \multicolumn{2}{c|}{$d=0.001$} \\
			\cline { 2 - 7 } & $\|e_u^{R}\|_{L^\infty(R,a)}$ & $\|e_v^{R}\|_{L^\infty(a,b)}$ & $\|e_u^{R}\|_{L^\infty(R,a)}$ & $\|e_v^{R}\|_{L^\infty(a,b)}$ & $\|e_u^{R}\|_{L^\infty(R,a)}$ & $\|e_v^{R}\|_{L^\infty(a,b)}$ \\
			\hline
   50 & 4.70{\rm e}-13 & 5.00{\rm e}-13 &  5.09{\rm e}-13 & 5.00{\rm e}-13  &    4.98{\rm e}-13 & 5.04{\rm e}-13  \\
   100 & 9.59{\rm e}-12 & 8.16{\rm e}-12 & 9.59{\rm e}-12 & 8.16{\rm e}-12  & 9.94{\rm e}-12 & 6.74{\rm e}-12 \\
   150 & 3.57{\rm e}-12 & 4.91{\rm e}-12  &  3.56{\rm e}-12 & 4.91{\rm e}-12   &  3.50{\rm e}-12 & 4.87{\rm e}-12  \\
			200 & 4.01{\rm e}-12& 2.09{\rm e}-12 & 4.01{\rm e}-12 & 2.09{\rm e}-12 & 4.00{\rm e}-12 & 3.22{\rm e}-12 \\
   \hline\multirow{2}{*}{$k$} & \multicolumn{2}{|c|}{$d=1$} & \multicolumn{2}{c|}{$d=0.1$} & \multicolumn{2}{c|}{$d=0.001$} \\
			\cline { 2 - 7 } & $\|e_u^{I}\|_{L^\infty(R,a)}$ & $\|e_v^{I}\|_{L^\infty(a,b)}$ & $\|e_u^{I}\|_{L^\infty(R,a)}$ & $\|e_v^{I}\|_{L^\infty(a,b)}$ & $\|e_u^{I}\|_{L^\infty(R,a)}$ & $\|e_v^{I}\|_{L^\infty(a,b)}$ \\
			\hline
   50&8.42{\rm e}-13 & 1.23{\rm e}-13  & 8.34{\rm e}-13 & 1.31{\rm e}-13 & 8.37{\rm e}-13 & 1.33{\rm e}-13 \\
   100 & 2.96{\rm e}-12 & 7.43{\rm e}-12 & 2.96{\rm e}-12& 7.43{\rm e}-12  & 2.99{\rm e}-12 & 7.44{\rm e}-12 \\
   150&1.90{\rm e}-12 & 2.12{\rm e}-12 &  1.92{\rm e}-12 & 2.26{\rm e}-12   &  1.88{\rm e}-12 & 3.02{\rm e}-12  \\
			200 & 1.58{\rm e}-12 & 2.68{\rm e}-12 & 1.58{\rm e}-12 & 2.68{\rm e}-12 & 1.59{\rm e}-12 & 1.73{\rm e}-12 \\
   \hline
		\end{tabular}
	\end{table}

	\section{Rectangular real compressed layer}\label{Sect3:RRCL}
	In this section, we  construct the rectangular RCL, which is more practical for  domain reduction of Helmholtz scattering problems with more general scatterers.
		\begin{figure}[!tb]
		\centering
		\subfigure[PML domain]{
			\includegraphics[width=0.25\textwidth]{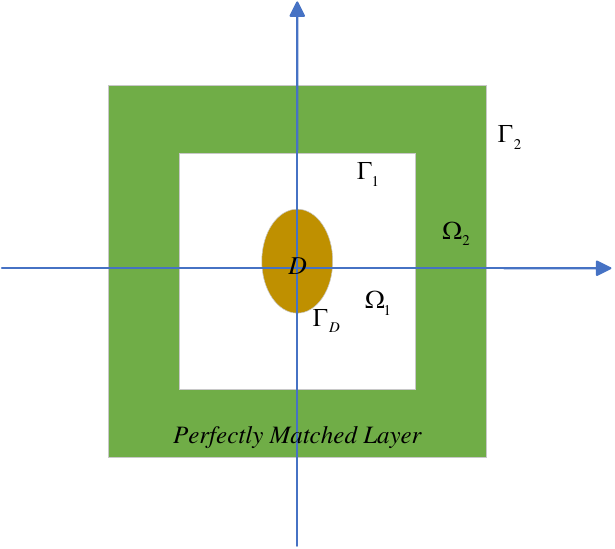}
		} \quad
		\subfigure[RCL domain]{		\includegraphics[width=0.25\textwidth]{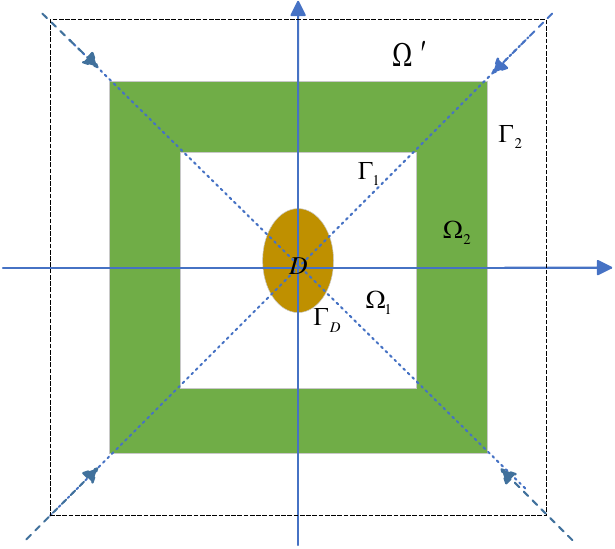}
  }
   \subfigure[Illustrate of RCL domain]{
	\includegraphics[width=0.25\textwidth]{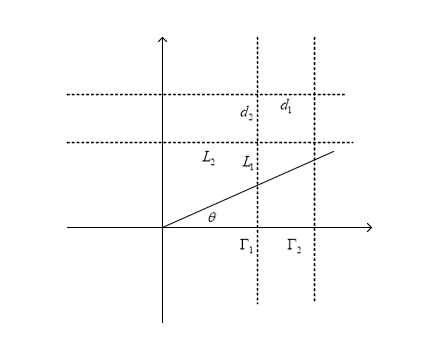}
		}
		\caption{Schematic illustrations of the rectangular PML and RCL with four trapezoidal patches, and the compression transformation from radial direction..}\label{mappingarea} 
	\end{figure}

	\subsection{The rectangular RCL equation} Let $R_1 = \{\hat{\bs x}\in \mathbb{R}^2:|\hat x_1|<L_1,|\hat x_2|<L_2\}$ be the  rectangular domain that
	 contains the scatterer $D.$  Let $\Gamma_1 = \partial R_1$ and  $\Omega_1 = R_1\setminus\bar{D}$ 
	  (see Figure \ref{mappingarea}). Note that $\Gamma_1$ has the parametric form:
	$	\rho=a(\theta),\, \theta \in[0,2\pi)$ with
	\begin{align}
		a(\theta) =
		\begin{cases}
			L_1\sec\theta, & \theta\in [0,\theta_0)\cup(2\pi-\theta_0,2\pi), \\
			L_2\csc\theta, & \theta\in [\theta_0,\pi-\theta_0), \\
			-L_1\sec\theta, & \theta\in [\pi-\theta_0,\pi+\theta_0), \\
			-L_2\csc\theta, & \theta\in [\pi+\theta_0,2\pi-\theta_0),
		\end{cases}\quad \theta_0 = \arctan\frac{L_2}{L_1},\label{mappingt}
	\end{align}
	where $(L_1,L_2),(-L_1,L_2),(L_1,-L_2),(-L_1,-L_2)$ are the four vertices of $R_1.$ 

	As with the circular case, the first step of designing the RCL technique  is to introduce an exponential mapping.
	 Similar to \eqref{rrmap}, we define the following coordinate transformation between $\rho = |\hat{\bs x}|$ and $r = |{{\bs x}}|:$
	\begin{equation}
		\rho = \tau(r,\theta) =
		\begin{cases}
			r, & r\leq a(\theta), \\
			a(\theta)e^{\tau_0(r-a(\theta))}, & r>a(\theta),
		\end{cases}\quad \theta = \theta,\label{rmappr}
	\end{equation}
but  $a$ is a function of $\theta$ in this context.
	Let $ u({{\bs x}}) = U(\hat{\bs x})$ be the solution of the scattering problem \eqref{helmholtz2eq2}--\eqref{helmholtz2bc2}, which satisfies $u=U$ on $\Gamma_1.$ Then using the  change of variables (see  \cite[Lemma 4.1]{yang2019truly}), we can transform the Helmholtz equation \eqref{helmholtz2eq2} into
	\begin{equation}
		-{\mathbb J}^{-1}\nabla\cdot\left(\bs{C}\nabla  u\right)-k^2  u = 0\quad\hbox{in}\;\;\; \mathbb{R}^2\setminus\bar{R}_1,\label{recrcl}
	\end{equation}
	where
	\begin{equation}\label{invjacobian1}
	{\bs C} = {\bs J}^{-1}{\bs J}^{-\top}\mathbb J,\quad	{\bs J}^{-1}=\frac{1}{\partial_r\tau}{\bs R}_{\theta}\begin{pmatrix}
			1 & -\partial_\theta\tau/\tau\\
			0 & r\partial_r\tau/\tau
		\end{pmatrix}{\bs  R}_{\theta}^{\top},\quad \mathbb J^{-1}=|\bs J^{-1}|=\frac{r}{\tau\partial_r\tau},
	\end{equation}
with ${\bs R}_{\theta}$ being the rotation matrix. Compared with the circular case in \eqref{RPMLeq}--\eqref{invjacobianr}, the matrix ${\bs J}$ in \eqref{invjacobian1} is  more complicated.

 Using \eqref{ascir}, we find that for any fixed $\theta$ and $\rho>a(\theta)$,
\begin{equation*}
		 u(r,\theta) := U(\rho,\theta)=\sqrt{\frac{2}{k\pi a}}\, e^{-\frac{\tau_0}{2}(r-a)}\,{\rm exp}\big({\rm i}k a\, e^{\tau_0(r-a)}\big)\, M(k\tau(r),\theta).
\end{equation*}
It means that {\em the solution  $u(r,\theta)$ of the transformed Helmholtz problem \eqref{recrcl} must  decay exponentially}. This motivates us to truncate \eqref{recrcl}  directly by a rectangular boundary and impose the  homogeneous Dirichlet boundary condition, which leads to the {\em  Helmholtz equation reduced by  the RCL technique} and  the corresponding  {\bf RCL-equation} below.
%
%
For this purpose, we first introduce some notation. As illustrated in Figure \ref{mappingarea} (left), we surround $\Omega_1$ by a rectangular layer:
 $$R_1\subset R_2: = \{{\bs x}\in \mathbb{R}^2:|{x}_1|\leq L_1+d_1,|{{x}}_2| \leq L_2+d_2\},$$
 and denote
 $$\Gamma_2 = \partial {R}_2, \;\;\; \Omega_2 = {B}_2\setminus \bar{R}_1, \;\;\;  \Omega = {R}_2\setminus D=\Omega_1\cup \Omega_2.$$
 Like $\Gamma_1,$ the outer rectangle $\Gamma_2$ has the parametric representation $r=b(\theta)$ with $L_i+d_i $ in place of $L_i, i=1,2$  in \eqref{mappingt}.
	Then, we obtain the {\bf RCL-equation:} 
	%
		\begin{subequations}\label{eqn:35}
			\begin{align}
				&-\mathbb J^{-1}\nabla\cdot\left(\bs{C}\nabla \hat{u}(\bs x)\right)-k^2 \hat{u}(\bs x) = 0 \hbox{ in }\Omega,\label{modelach1}\\
				& \hat{u} = g \hbox{ on }\Gamma_D,\quad \hat{u} = 0 \hbox{ on }\Gamma_2,\label{modelach2}
			\end{align}
		\end{subequations}
		where $\bs{C}=\bs{I},\mathbb J^{-1}=1$ in $\Omega_1$ and $\bs{C} = \bs J^{-1}\bs J^{-\top}\mathbb J$ as defined in \eqref{invjacobian1} in $\Omega_2.$
\begin{remark} {\em Note that we follow the same principle for constructing the RCL for both . Since the transformations used in the rectangular RCL equations are angle-dependent, this makes the equations for the rectangular RCL slightly more complex than the circular RCL equations.}	
\end{remark}


	\subsection{Convergence analysis}
	We start with introducing the Dirichlet-to-Neumann (DtN) operator $T: H^{1/2}(\Gamma_1)\rightarrow H^{-1/2}(\Gamma_1):$
	\begin{equation}
		T \xi = \frac{\partial \xi}{\partial {\bs n}_1}\quad \hbox{on }\;\; \Gamma_1,\nonumber
	\end{equation}
	where ${\bs n}_1$ is the unit outer normal vector  along $\Gamma_1,$ and
	 $\xi$ is the solution of the exterior  problem:
	\begin{subequations}
		\begin{align}
			&-\Delta\xi-k^2\xi = 0\quad\hbox{ in }\;\; \Omega_{ex}:=\mathbb{R}^2\setminus \bar{R}_1,\label{exhel1}\\
			&\xi = \zeta\quad\hbox{ on }\;\; \Gamma_1,\label{exhel2}\\
			&\frac{\partial  \xi} {\partial \rho} -\ri k \xi= o(\rho^{-1/2}) \quad {\rm as}\;\; \rho = |{\bs x}| \to \infty.\label{exhel3}
		\end{align}
	\end{subequations}
	It is known that \eqref{exhel1}--\eqref{exhel3} has a unique solution $\xi\in H_{\rm loc}^1(\Omega_{ex})$ (see e.g., \cite{colton2013integral}) and  the DtN map   $T: H^{1/2}(\Gamma_1)\rightarrow H^{-1/2}(\Gamma_1)$ is  a well-defined,  continuous linear operator.
 We further  introduce the following single and double layer potentials:
	\begin{subequations}
		\begin{align}
			&{\Psi}_{\rm SL}^k(\lambda)(\hat{\bs x})=\int_{\Gamma_1}{G}_k(\hat{\bs x},\hat{\bs y})\lambda(\hat{\bs y})\, ds(\hat{\bs y}),\quad\forall \lambda\in H^{-1/2}(\Gamma_1),\label{modifgreen1}\\
			&{\Psi}_{\rm DL}^k(\zeta)(\hat{\bs x}) = \int_{\Gamma_1}\frac{\partial {G}_k(\hat{\bs x},\hat{\bs y})}{\partial {\bs n}_1(\hat{\bs y})}\zeta(\hat{\bs y})\, ds(\hat{\bs y}),\quad\forall \zeta\in H^{1/2}(\Gamma_1),\label{modifgreen2}
		\end{align}
	\end{subequations}
	where
\begin{equation}
	{G}_k(\hat{\bs x},\hat{\bs y}) = \frac{{\rm i}}{4}H_0^{(1)}(k|\hat{\bs x}-\hat{\bs y}|)\label{modifiedgreen}
\end{equation}
is the fundamental solution of the  Helmholtz equation with the Sommerefeld radiation boundary condition and $\lambda = \frac{\partial \zeta}{\partial {\bs n}_1(\bs y)} \in H^{-\frac{1}{2}}(\Gamma_1)$ is the Neumann trace of $\zeta$ on $\Gamma_1$. It is easy to see that ${G}_k$ is smooth for $\hat{\bs x}\in \Omega_{ex},$ and we have
	\begin{equation}
		U(\hat{\bs x})= -{\Psi}_{\rm SL}^k(\lambda)(\hat{\bs x})+{\Psi}_{\rm DL}^k(\zeta)(\hat{\bs x})\quad \hbox{in} \;\; \Omega_{ex}.\label{orgsolutionhel}
	\end{equation}
The following estimates of the Green's function are useful in the forthcoming analysis. 	
	\begin{lemma}\label{lemmagreen}
   For any $\hat{\bs x}\in \Omega_{ex}, \hat{\bs y}\in \Gamma_1$ with $k|\hat{\bs x}-\hat{\bs y}|\ge 1$, we have
		\begin{subequations}
			\begin{align}
				&\left|{G}_k(\hat{\bs x},\hat{\bs y})\right|\leq \frac{1}{4} k^{-1/2}|\hat{\bs x}-\hat{\bs y}|^{-1/2}|H_1^{(1)}(1)|;\label{Greenesti1}\\
				&\left|\frac{\partial {G}_k(\hat{\bs x},\hat{\bs y})}{\partial \hat{x}_j}\right|\leq  Ck |\hat{\bs x}-\hat{\bs y}|^{-1/2}|H_1^{(1)}(1)|;\label{Greenesti2}\\
				&\left|\frac{\partial {G}_k(\hat{\bs x},\hat{\bs y})}{\partial \hat{y}_j}\right|\leq  Ck |\hat{\bs x}-\hat{\bs y}|^{-1/2}|H_1^{(1)}(1)|;\label{Greenesti3}\\
				&\left|\frac{\partial^2 {G}_k(\hat{\bs x},\hat{\bs y})}{\partial \hat{x}_i\partial \hat{y}_j}\right|\leq  C\big(k^{1/2}|\hat{\bs x}-\hat{\bs y}|^{-3/2}+k^{3/2}|\hat{\bs x}-\hat{\bs y}|^{-1/2}\big)|H_1^{(1)}(1)|,\label{Greenesti4}
			\end{align}
		\end{subequations}
		for $i,j=1,2,$ where $C$ is a positive constant independent of $\hat{\bs x},\hat{\bs y}$ and $k.$
	\end{lemma}
	\begin{proof}
		 By  \eqref{hankelk021},
		\begin{equation}
			|G_k(\hat{\bs x},\hat{\bs y})|^2 = \frac{1}{16}\big|H_0^{(1)}(k|\hat{\bs x}-\hat{\bs y}|)\big|^2\leq \frac{1}{16}|H_1^{(1)}(k|\hat{\bs x}-\hat{\bs y}|)|^2.\nonumber
		\end{equation}
Using the fact that $z|H_n^{(1)}(z)|^2$ is a strictly decreasing function of $z$ (cf.   \cite[p.\!\! 446]{watson1995treatise}), we have
		\begin{equation}
			|G_k(\hat{\bs x},\hat{\bs y})|^2 \leq \frac{1}{16}\frac{k|\hat{\bs x}-\hat{\bs y}|\big|H_1^{(1)}(k|\hat{\bs x}-\hat{\bs y}|)\big|^2}{|H_1^{(1)}(1)|^2}\frac{|H_1^{(1)}(1)|^2}{k|\hat{\bs x}-\hat{\bs y}|}\leq \frac{1}{16k|\hat{\bs x}-\hat{\bs y}|}|H_1^{(1)}(1)|^2,\nonumber
		\end{equation}
		which yields \eqref{Greenesti1}. Using the formula
		\begin{equation}
			\frac{d H_0^{(1)}(z)}{dz} = - H_1^{(1)}(z),\nonumber
		\end{equation}
		we derive
		\begin{equation}
			\frac{\partial {G_k}(\hat{\bs x},\hat{\bs y})}{\partial \hat{x}_j} = -\frac{\rm i}{4}kH_1^{(1)}(k|\hat{\bs x}-\hat{\bs y}|)\frac{\hat{x}_j-\hat{y}_j}{|\hat{\bs x}-\hat{\bs y}|}.\label{derG}
		\end{equation}
		Hence,
		\begin{eqnarray}
			\Big|\frac{\partial {G_k}(\hat{\bs x},\hat{\bs y})}{\partial \hat{x}_j}\Big|^2\leq \frac{k^2}{16}\big|H_1^{(1)}(k|\hat{\bs x}-\hat{\bs y}|)\big|^2\leq\frac{k^2}{16|\hat{\bs x}-\hat{\bs y}|}|H_1^{(1)}(1)|^2.\nonumber
		\end{eqnarray}
		Then we obtain \eqref{Greenesti2}. Similarly, we can prove \eqref{Greenesti3}.
		
		To prove \eqref{Greenesti4}, we recall the formula
		\begin{equation}
			z\frac{dH_n^{(1)}(z)}{dz}+nH_n^{(1)}(z)=zH_{n-1}^{(1)}(z).\nonumber
		\end{equation}
		Hence, by \eqref{derG},
		\begin{equation}
			\begin{split}
				\frac{\partial^2 {G_k}(\hat{\bs x},\hat{\bs y})}{\partial \hat{x}_i\partial \hat{y}_j} &= -\frac{{\rm i}}{4}kH_1^{(1)}(k|\hat{\bs x}-\hat{\bs y}|)\Big(\frac{|\hat{\bs x}-\hat{\bs y}|^2\delta_{ij}-(\hat{x}_i-\hat{y}_i)(\hat{x}_j-\hat{y}_j)}{|\hat{\bs x}-\hat{\bs y}|^3}\Big)\nonumber\\
				&-\frac{{\rm i}}{4}k^2\frac{(\hat{x}_j-\hat{y}_j)(\hat{x}_i-\hat{y}_i)}{|\hat{\bs x}-\hat{\bs y}|^2}\Big(H_0^{(1)}(k|\hat{\bs x}-\hat{\bs y}|)-\frac{1}{k|\hat{\bs x}-\hat{\bs y}|}H_1^{(1)}(k|\hat{\bs x}-\hat{\bs y}|)\Big).\nonumber
			\end{split}
		\end{equation}
		Therefore, we have the following estimate
		\begin{equation}
			\begin{split}
				\Big|\frac{\partial^2 {G_k}(\hat{\bs x},\hat{\bs y})}{\partial \hat{x}_i\partial \hat{y}_j}\Big|
				&\leq Ck\frac{1}{|\hat{\bs x}-\hat{\bs y}|}H_1^{(1)}(k|\hat{\bs x}-\hat{\bs y})|+k^2|H_0^{(1)}(k|\hat{\bs x}-\hat{\bs y}|)|\nonumber\\
				&\leq Ck^{1/2}|\hat{\bs x}-\hat{\bs y}|^{-3/2}|H_1^{(1)}(1)|+Ck^{3/2}|\hat{\bs x}-\hat{\bs y}|^{-1/2}|H_1^{(1)}(1)|.\nonumber
			\end{split}
		\end{equation}
		This completes the proof.
	\end{proof}

	Now, we are in position to estimate the single and double layer potentials $\Psi_{\rm SL}^k,\Psi_{\rm DL}^k$. In what follows, we shall use the following weighted $H^{1/2}(\Gamma)$ norm (see \cite{chen2010convergence}),
	\begin{equation}
		\|w\|_{H^{1/2}(\Gamma)} = \Big(|\Gamma|^{-1}\|w\|^2_{L^2(\Gamma)}+|\Gamma|^{-1}|w|^2_{H^{1/2}(\Gamma)}\Big)^{1/2},\label{12norm}
	\end{equation}
	where
	\begin{equation}		|w|_{H^{1/2}(\Gamma)}^2=\int_\Gamma\int_\Gamma\frac{|w({\bs x})-w({\bs x}')|^2}{|{\bs x}-{\bs x}'|^2}ds({\bs x})ds({\bs x}').\label{defh12}	
	\end{equation}
	\begin{lemma}\label{lemmasingle}
  Let  $\Gamma$ be a piece-wise smooth, simply connected, closed curve in $\Omega_{\rm ex}$ satisfying  $k\, {\rm dist}(\Gamma; \Gamma_1)>1$  and $k\ge k_0>0$. Then for any $\lambda\in H^{-1/2}(\Gamma_1)$, we have 
		\begin{equation}
	\|\Psi_{\rm SL}^k(\lambda)(\hat{\bs x})\|_{H^{1/2}(\Gamma)}\leq Ck^{\frac{3}{2}} \big\{{\rm dist}(\Gamma; \Gamma_1)\big\}^{-\frac 12}\|\lambda\|_{H^{-1/2}(\Gamma_1)},\label{estisinglelayer}
		\end{equation}
  where $C$ is a positive constant independent of $k$ and $\lambda,$ and $\Psi_{\rm SL}^k(\lambda)$ is the single layer potential defined in \eqref{modifgreen1}.
	\end{lemma}
	\begin{proof}
		By the definition \eqref{modifgreen1}, we obtain from the dual norm estimate and trace theorem that
  \begin{equation*}
      |\Psi_{\rm SL}^k(\lambda)(\hat{\bs x})|\leq \|\lambda\|_{H^{-1/2}(\Gamma_1)}\|G_k(\hat{\bs x},\cdot)\|_{H^{1/2}(\Gamma_1)}\leq C
      \|\lambda\|_{H^{-1/2}(\Gamma_1)}\|G_k(\hat{\bs x},\cdot)\|_{H^{1}(\Omega_1)}.
  \end{equation*}
 Using the embedding property and Lemme \ref{lemmagreen}, we find
  \begin{equation*}
  \begin{split}
      \|G_k(\hat{\bs x},\cdot)\|_{H^{1}(\Omega_1)}&\leq C|\Omega_1|^{1/2}(\|G_k(\hat{\bs x},\cdot)\|_{L^{\infty}(\Omega_1)}+\|\nabla_{\hat{\bs y}}G_k(\hat{\bs x},\cdot)\|_{L^{\infty}(\Omega_1)})\\
      &\le Ck(1+k^{-3/2}) \||\hat{\bs x}- \hat{\bs y}|^{-1/2}\|_{L^{\infty}(\Omega_1)}\\
      &\le Ck  \big\{{\rm dist}(\hat{\bs x}; \Gamma_1)\big\}^{-\frac 12}.
      \end{split}
  \end{equation*}
Thus,  we have
$$
|\Psi_{\rm SL}^k(\lambda) (\hat{\bs x})|\le Ck  \big\{{\rm dist}(\hat{\bs x}; \Gamma_1)\big\}^{-\frac 12}  \|\lambda\|_{H^{-1/2}(\Gamma_1)}.
$$
		Hence, we can derive
		\begin{equation}
  \begin{split}
|\Gamma|^{-1/2}\|\Psi_{\rm SL}^k(\lambda)  (\cdot) \|_{L^2(\Gamma)} &\leq C\|\Psi_{\rm SL}^k(\lambda)  (\cdot) \|_{L^\infty(\Gamma)}\\
&\leq C k \sup_{ \hat{\bs x} \in\Gamma} \big\{{\rm dist}(\hat{\bs x}; \Gamma_1)\big\}^{-\frac 12}\|\lambda\|_{H^{-1/2}(\Gamma_1)}\\
&=C k  \big\{{\rm dist}(\Gamma; \Gamma_1)\big\}^{-\frac 12}\|\lambda\|_{H^{-1/2}(\Gamma_1)}.\label{essl1}
\end{split}
		\end{equation}
		It remains to estimate $|\Psi_{\rm SL}^k(\lambda)|_{H^{1/2}(\Gamma)}$. Using the Mean Value Theorem, we know for any $\hat{\bs x},\hat{\bs x}'\in \Gamma,$
		\begin{equation}
			|\Psi_{\rm SL}^k(\lambda)(\hat{\bs x})-\Psi_{\rm SL}^k(\lambda)(\hat{\bs x}')| \leq C\|\nabla_{\hat{\bs x}}\Psi_{SL}(\lambda)\|_{L^\infty(\Gamma)}|\hat{\bs x}-\hat{\bs x}'|.\nonumber
		\end{equation}
Again from  the definition \eqref{modifgreen1}, we obtain from the dual norm estimate and trace theorem that
\begin{equation*}
|\nabla_{\hat{\bs x}}\Psi_{\rm SL}^k(\lambda)(\hat{\bs x})|\leq \|\lambda\|_{H^{-1/2}(\Gamma_1)}\|\nabla_{\hat{\bs x}}G_k(\hat{\bs x},\cdot)\|_{H^{1/2}(\Gamma_1)}\leq C
      \|\lambda\|_{H^{-1/2}(\Gamma_1)}\|\nabla_{\hat{\bs x}}G_k(\hat{\bs x},\cdot)\|_{H^{1}(\Omega_1)}.
\end{equation*}
 Using the embedding property and Lemme \ref{lemmagreen}, we find
  \begin{equation*}
  \begin{split}
      \|\nabla_{\hat{\bs x}}G_k(\hat{\bs x},\cdot)\|_{H^{1}(\Omega_1)}&\leq C|\Omega_1|^{\frac 12}(\|\nabla_{\hat{\bs x}}G_k(\hat{\bs x},\cdot)\|_{L^{\infty}(\Omega_1)}+\|\nabla_{\hat{\bs y}}\nabla_{\hat{\bs x}}G_k(\hat{\bs x},\cdot)\|_{L^{\infty}(\Omega_1)})\\
      &\leq  Ck \||\hat{\bs x}- \hat{\bs y}|^{-1/2}\|_{L^{\infty}(\Omega_1)}+  C\big(k^{\frac 12}\||\hat{\bs x}- \hat{\bs y}|^{-3/2}\|_{L^{\infty}(\Omega_1)}
      \\
      &\quad +k^{\frac 3 2}\||\hat{\bs x}- \hat{\bs y}|^{-1/2}\|_{L^{\infty}(\Omega_1)}\big)\\
      &\le Ck^{\frac{3}{2}}(1+k^{-\frac{1}{2}})\||\hat{\bs x}- \hat{\bs y}|^{-1/2}\|_{L^{\infty}(\Omega_1)}+Ck^{\frac{1}{2}} \||\hat{\bs x}- \hat{\bs y}|^{-3/2}\|_{L^{\infty}(\Omega_1)}\\
      &\le Ck^{\frac{3}{2}}  \big\{{\rm dist}(\hat{\bs x}; \Gamma_1)\big\}^{-\frac 12},
      \end{split}
  \end{equation*}
  where we used $k\, |\hat{\bs x}- \hat{\bs y}|>1.$
Thus using \eqref{defh12}  and the above two bounds, we  obtain
		\begin{equation}
			\begin{split}
				|\Gamma|^{-1}|\Psi_{\rm SL}(\lambda)(\hat{\bs x})|_{H^{1/2}(\Gamma)}&\leq C\|\nabla_{\hat{\bs x}}\Psi_{\rm SL}(\lambda)\|_{L^\infty(\Gamma)}\\
                &\leq Ck^{\frac{3}{2}}\sup_{ \hat{ x} \in\Gamma} \big\{{\rm dist}(\hat{\bs x}; \Gamma_1)\big\}^{-\frac 12}\|\lambda\|_{H^{-1/2}(\Gamma_1)}\\
&=C k^{\frac{3}{2}} \big\{{\rm dist}(\Gamma; \Gamma_1)\big\}^{-\frac 12}\|\lambda\|_{H^{-1/2}(\Gamma_1)}.\label{essl2}
			\end{split}
		\end{equation}
	Then the desired bound is a direct consequence of  \eqref{essl1}--\eqref{essl2}.
	\end{proof}
	
	\begin{lemma}\label{lemmadouble}
Under the same conditions as in Lemma \ref{lemmasingle}, we
have that for any $\zeta\in H^{1/2}(\Gamma_1)$,
		\begin{equation}
		\|\Psi_{\rm DL}^k(\zeta)(\hat{\bs x})\|_{H^{1/2}(\Gamma)}\leq  C k^{\frac{3}{2}}\big\{{\rm dist}(\Gamma; \Gamma_1)\big\}^{-\frac 12}\|\zeta\|_{H^{1/2}(\Gamma_1)},\label{estidoublelayer}
		\end{equation}
  where  $C$ is a positive constant independent of $k$ and $\zeta,$ and $\Psi_{\rm DL}^k(\zeta)$ is the double layer potential defined in \eqref{modifgreen2}.
	\end{lemma}
	\begin{proof}
By \eqref{modifgreen2} and Lemma \ref{lemmagreen}, we have
		\begin{equation}	|\Psi_{\rm DL}^k(\zeta)(\hat{\bs x})|\leq\|\partial_{{\bs n}_1(\hat{\bs y})}G_k(\hat{\bs x},\cdot)\|_{L^\infty(\Gamma_1)}\|\zeta\|_{L^1(\Gamma_1)}\leq Ck\big\{{\rm dist}(\hat{\bs x}; \Gamma_1)\big\}^{-\frac 12}\|\zeta\|_{L^1(\Gamma_1)}.\nonumber
		\end{equation}
		Hence,  we obtain
		\begin{equation}
        \begin{split}
|\Gamma|^{-1/2}\|\Psi_{\rm DL}^k(\zeta)\|_{L^2(\Gamma)}&\leq \|\Psi_{\rm DL}^k(\zeta)\|_{L^\infty(\Gamma)}\leq  Ck\sup_{ \hat{ x} \in\Gamma} \big\{{\rm dist}(\hat{\bs x}; \Gamma_1)\big\}^{-\frac 12}\|\zeta\|_{L^1(\Gamma_1)}\\
&\leq Ck\big\{{\rm dist}(\Gamma; \Gamma_1)\big\}^{-\frac 12}\|\zeta\|_{L^1(\Gamma_1)}.\label{doubeles1}
        \end{split}
		\end{equation}
		Using the Mean Value Theorem, we know that for any $\hat{\bs x},\hat{\bs x}'\in \Gamma,$
		\begin{equation}
			|\Psi_{\rm DL}^k(\zeta)(\hat{\bs x})-\Psi_{\rm DL}^k(\zeta)(\hat{\bs x}')| \leq C\|\nabla_{\hat{\bs x}}\Psi_{\rm DL}^k(\zeta)\|_{L^\infty(\Gamma)}|\hat{\bs x}-\hat{\bs x}'|.\nonumber
		\end{equation}
  By the definition of $\Psi_{\rm DL}^k(\zeta)(\hat{\bs x})$, we obtain from Lemma \ref{lemmagreen} that
  \begin{equation*}	
  \begin{split}
        |\nabla_{\hat{\bs x}}\Psi_{\rm DL}^k(\zeta)(\hat{\bs x})|&\leq\|\nabla_{\hat{\bs x}}\partial_{{\bs n}_1(\hat{\bs y})}G_k(\hat{\bs x},\cdot)\|_{L^\infty(\Gamma_1)}\|\zeta\|_{L^1(\Gamma_1)}\\
        &\leq Ck^{\frac{3}{2}}\big\{{\rm dist}(\hat{\bs x}; \Gamma_1)\big\}^{-\frac 12}\|\zeta\|_{L^1(\Gamma_1)}+ Ck^{\frac{1}{2}}\big\{{\rm dist}(\hat{\bs x}; \Gamma_1)\big\}^{-\frac 32}\|\zeta\|_{L^1(\Gamma_1)}\\
                &\leq Ck^{\frac{3}{2}}\big\{{\rm dist}(\Gamma; \Gamma_1)\big\}^{-\frac 12}\|\zeta\|_{L^1(\Gamma_1)}.
  \end{split}
		\end{equation*}
		Thus using the definition of $|\Psi_{\rm DL}^k(\zeta)(\hat{\bs x})|_{H^{{1}/{2}}(\Gamma)}$, we have
		\begin{equation}
			\begin{split}
				|\Gamma|^{-1}|\Psi_{\rm DL}^k(\zeta)|_{H^{{1}/{2}}(\Gamma)}&\leq C\|\nabla_{\hat{\bs x}}\Psi_{\rm DL}^k(\zeta)\|_{L^\infty(\Gamma)}\leq Ck^{\frac{3}{2}}\sup_{ \hat{ x} \in\Gamma} \big\{{\rm dist}(\hat{\bs x}; \Gamma_1)\big\}^{-\frac 12}\|\zeta\|_{L^1(\Gamma_1)}\\
                &\leq Ck^{\frac{3}{2}}\big\{{\rm dist}(\Gamma; \Gamma_1)\big\}^{-\frac 12}\|\zeta\|_{L^1(\Gamma_1)}.\label{doubeles2}
			\end{split}
		\end{equation}
		On the other hand, using the embedding theorem and the definition of $\|\Psi_{\rm DL}^k(\zeta)\|_{H^{1/2}(\Gamma)}$, we arrive at
		\begin{equation}
			\|\zeta\|_{L^1(\Gamma_1)}\leq C\|\zeta\|_{L^2(\Gamma_1)}\leq C\|\zeta\|_{H^{1/2}(\Gamma_1)}.\label{doubeles3}
		\end{equation}
A combination of \eqref{doubeles1}-\eqref{doubeles3} completes the proof.
	\end{proof}

 With the above preparations, we are now ready to conduct the convergence analysis of the rectangular RCL.
\begin{thm}\label{conv-D}
		The RCL-equation \eqref{eqn:35} has a unique solution $\hat{u}\in H^1(\Omega)$, which converges to the solution $U$ of
  the original scattering problem \eqref{helmholtz2eq}  exponentially in $\Omega_1:$
			\begin{equation}
		\|U-\hat{u}\|_{H^1(\Omega_1)}=\|u-\hat{u}\|_{H^1(\Omega_1)}  \leq {\mathcal  C}  e^{-\frac{1}{2}\tau_0\min\{d_1,d_2\}}\|U\|_{H^{1/2}(\Gamma_1)},\label{convapml1}
			\end{equation}
   where ${\mathcal C}$ is a positive constant depending on $k,$ and $\tau_0, d_1, d_2$ are defined as before. 
	\end{thm}
	\begin{proof}
 We transform  the RCL equation \eqref{modelach1}-\eqref{modelach2} with the solution $\hat{u}(x,y)$  back to the $(\hat{x},\hat{y})$ coordinates through \eqref{rmappr}, leading to
		\begin{equation} \label{RCLre}
			-{\Delta}\widehat{U}-k^2\widehat{U} = 0 \;\; \hbox{ in}\;\; \Omega';\quad
			\widehat{U} = g\;  \hbox{ on }\; \partial D; \quad \widehat{U} =0\; \hbox{ on }\; \Gamma_{\hat{b}(\theta)},
		\end{equation}
		where  $\Omega'$ is the image of $ \Omega$ under the mapping $\rho=\tau(r, \theta),$ and
\begin{equation}\label{RCLre_map}
\widehat{U}(\rho,\theta)=\hat u(\tau^{-1}(\rho,\theta),\theta),\quad  \hat{b}(\theta)=\tau(b(\theta),\theta),\quad \Gamma_{\hat{b}(\theta)}=\big\{\hat{\bs x}\in \mathbb{R}^2\,:\,|\hat{\bs x}|=\hat{b}(\theta)\big\}.
  \end{equation}

  Assume that $k^2$ is not a Dirichlet eigenvalue of $-\Delta$. Then the problem \eqref{RCLre} has a unique solution $\widehat{U}(\hat{\bs x})\in H^1(\Omega')$ such that (see e.g., \cite{griesmaier2011error,mclean2000strongly})
		\begin{equation}\label{regarOmega}
			\|\widehat{U}\|_{H^1(\Omega')}\leq {\mathcal C}\|g\|_{H^{1/2}(\Gamma_1)}, 
		\end{equation}
		where the positive constant ${\mathcal C}$ depends on $k$ and $\Omega'$. This implies the existence and uniqueness of the solution to the problem \eqref{modelach1}-\eqref{modelach2}.

 Suppose that $q_1$ be the solution of the original scattering problem \eqref{helmholtz2eq} on $\Gamma_1$, i.e., $q_1(\hat{\bs x})=U|_{\hat{\bs x}\in \Gamma_1}$. Then from the solution formula \eqref{orgsolutionhel}, we know that the solution of the original scattering problem \eqref{helmholtz2eq} on the outer rectangle  $\Gamma_{\hat{b}(\theta)}$ is
 \begin{equation*}
     q_2(\hat{\bs x}) = U|_{\hat{\bs x}\in \Gamma_{\hat{b}(\theta)}}=-\Psi_{\rm SL}^k(\lambda)(\hat{\bs x})+\Psi_{\rm DL}^k(\zeta)(\hat{\bs x}),\quad \hat{\bs x}\in \Gamma_{\hat{b}(\theta)},
 \end{equation*}
 where $ \lambda=\frac{\partial q_1}{\partial {\bs n}_y}$ and $  \zeta=q_1$ on $\Gamma_1$. Accordingly the original scattering problem \eqref{helmholtz2eq} enclosed by $\Gamma_{\hat{b}(\theta)}$ becomes
 \begin{equation}\label{equieq1}
     -{\Delta}{U}-k^2{U} = 0 \;\; \hbox{ in}\;\; \Omega';\quad
	{U} = g\;  \hbox{ on }\; \partial D; \quad {U} =q_2\; \hbox{ on }\; \Gamma_{\hat{b}(\theta)}.
 \end{equation}

Letting $e = U - \widehat{U}$ and  subtracting \eqref{RCLre} from \eqref{equieq1}, we obtain the following error equation:
		\begin{equation} \label{errorARCMLlayerorgin}
			-{\Delta}e-k^2e = 0\;\; \hbox{in }\;\; \Omega';\quad
			e = 0 \;\;  \hbox{ on } \;\; \partial D; \quad e = q_2\;\; \hbox{ on } \;\; \Gamma_{\hat {b}(\theta)}.
		\end{equation}
Thus, using the regularity result \eqref{regarOmega} yields
		\begin{align}
			\|e\|_{H^1(\Omega')}\leq {\mathcal C}\|q_2\|_{H^{1/2}(\Gamma_{\hat{b}(\theta)})}.\label{errorarcl1}
		\end{align}
On the other hand, according to \cite{chen2010convergence}, the so-defined $\zeta$ and $\lambda$ satisfy
		\begin{equation}
			\|\lambda\|_{H^{-1/2}(\Gamma_1)}\leq {\mathcal C}\|\zeta\|_{H^{1/2}(\Gamma_1)}.\label{errorarcl2}
		\end{equation}
From the definition \eqref{12norm}, Lemma \ref{lemmasingle}, Lemma \ref{lemmadouble} and \eqref{errorarcl2}, we derive
\begin{equation}
    \|q_2\|_{H^{1/2}(\Gamma_{\hat{b}(\theta)})} \leq C {\mathcal C} k^{\frac{3}{2}}\big\{{\rm dist}(\Gamma_{\hat{b}(\theta)}; \Gamma_1)\big\}^{-\frac 12}\|\zeta\|_{H^{1/2}(\Gamma_1)},\label{estsquare}
\end{equation}
when $k\,{\rm dist}(\Gamma_{\hat{b}(\theta)}; \Gamma_1)>1.$
In view of the definition of the mapping $\tau(r,\theta)$ in \eqref{mappingt}-\eqref{rmappr}, we have
\begin{equation}
    {\rm dist}(\Gamma_{\hat{b}(\theta)}; \Gamma_1) = e^{\tau_0\min\{d_1,d_2\}},\label{estsquare1}
\end{equation}
as shown in Figure \ref{mappingarea}(c).

Finally, using \eqref{errorarcl1}, \eqref{estsquare} and \eqref{estsquare1}, we get
\begin{equation}
			\|e\|_{H^1(\Omega')}\leq  C {\mathcal C} k^{\frac{3}{2}} e^{-\frac{1}{2}\tau_0\min\{d_1,d_2\}}\|U\|_{H^{1/2}(\Gamma_1)}.\label{erorARCL1}
		\end{equation}
Hence
\begin{align}
	\|u-\hat{u}\|_{H^1(\Omega_1)} \leq \|e\|_{H^1(\Omega')} \leq Ce^{-\frac{1}{2}\tau_0\min\{d_1,d_2\}}\|U\|_{H^{1/2}(\Gamma_1)}.\nonumber
\end{align}
Note that $\tau(r,\theta)$ is an identical mapping  in $\Omega_1,$ so we have $u=U,$  $\hat u=\widehat U$ in $\Omega_1.$
This completes the proof.
  \end{proof}

 The proof of the above theorem implies the following estimate in the layer.
  \begin{corollary}
Under the mapping \eqref{rmappr}, we have the estimate in the RCL layer
$\Omega_2:$
			\begin{equation}
				\|\bs{C}^{-1}\nabla(u-\hat{u})\|_{L^2(\Omega_2)}+\|\mathbb J^{1/2}(u-\hat{u})\|_{L^2(\Omega_2)}\leq Ce^{-\frac{1}{2}\tau_0\min\{d_1,d_2\}}\|U\|_{H^{1/2}(\Gamma_1)},\label{convapml2}
			\end{equation}
   where $\bs{C}$ and $\mathbb J$ are the same as in the RCL-equation \eqref{eqn:35}.
  \end{corollary}
  \begin{proof}
Thanks to \eqref{erorARCL1}, we know
\begin{equation}
\|e\|_{H^1(\Omega')}\leq  C{\mathcal C} k^{\frac{3}{2}} e^{-\frac{1}{2}\tau_0\min\{d_1,d_2\}}\|U\|_{H^{1/2}(\Gamma_1)}.\nonumber
\end{equation}
Recall the notation: $u=U,$  $\hat u=\widehat U$ in $\Omega$ under the mapping \eqref{rmappr}. With a change of coordinates,
we obtain
\begin{align}
\|e\|_{H^1(\Omega'\setminus\Omega_1)}^2=
\|U-\widehat U\|_{H^1(\Omega'\setminus\Omega_1)}^2 =\|\bs{C}^{-1}\nabla(u-\hat{u})\|_{L^2(\Omega_2)}^2+\|\mathbb J^{1/2}(u-\hat{u})\|_{L^2(\Omega_2)}^2.\nonumber
\end{align}
This ends the proof.
\end{proof}
	
	\begin{remark} {\em It is seen from
  Theorem \ref{cirpcl} and Theorem \ref{conv-D} that
  the exponential factors are essentially of the same form but derived from different means. As in Remark  \ref{Rmk:22},   we can choose $\tau_0\min\{d_1,d_2\}=O(|\ln \epsilon|),$ for fixed $k$  and given accuracy threshold $\epsilon>0.$ Note that the condition for \eqref{estsquare} is automatically fulfilled.
 }
\end{remark}
	
	\subsection{Performance of the rectangular RCL}  In what follows, we use the finite element method to discretize the RCL-equation and demonstrate the performance of the proposed technique.

 \subsubsection{FEM implementation}
	A weak formulation of the RCL equation \eqref{eqn:35} is to find $\hat{u}\in H^1(\Omega)$, $\hat{u}=g$ on $\Gamma_D$ and $\hat{u}=0$ on $\Gamma_2$ such that
	\begin{align}
		\mathcal{B}(\hat{u},\psi) = \big(\bs{C}\nabla \hat{u},\nabla(r\tau^{-1}(r,\theta)\psi)\big)_{\Omega}-k^2\big(\partial_r\tau\,\hat{u},\psi\big)_{\Omega} = 0,\quad \forall \psi\in H_0^1(\Omega),\label{weakforanshel}
	\end{align}
 where $H_0^1(\Omega)=\{\psi\in H^1(\Omega):\psi|_{\partial\Omega}=0\}.$
	
	As it is shown in the circular case, the key to  success of the new approach is to introduce a suitable substitution for the unknown to diminish the oscillation near $\Gamma_1$: $r = a(\theta)$ in $\Omega_2$.
 For any fixed $\theta$ (see Figure \ref{mappingarea}) and $\rho>a(\theta)$,   we find from the the Karp's far field expansion in \eqref{karpsolution} and \eqref{solu} that under the RCT \eqref{rmappr}, the mapped field $u(r,\theta)=U(\tau(r,\theta),\theta)$ in the new coordinates decays exponentially (resulted from the factor $O(1/\sqrt \rho)$) but with an explicit oscillatory factor: $e^{{\rm i}k\rho} = e^{{\rm i}k(\tau(r,\theta)-a(\theta))}.$ This motivates us again to write the solution of the RCL-equation as  $\hat{u} = \hat{v}\,e^{{\rm i}k(\tau(r,\theta)-a(\theta))},$ and then solve for $\hat v,$
 which is expected to be free of oscillations.
  In view of this, we  reformulate \eqref{weakforanshel} as: Find $\hat{u} = \omega \hat{v}$ with $\hat{v}\in H^1(\Omega)$, $\hat{v} = g $ on $\Gamma_D$ and $\hat{v}=0$ on $\Gamma_2$ such that
	\begin{align}
		\widehat{\mathcal{B}}(\hat{v},\phi)=\mathcal{B}(\omega\hat{v},\omega\phi) = \Big(\bs{C}\nabla (\omega \hat{v}),\nabla\big(\frac{r\omega}{\tau(r,\theta)}\phi\big)\Big)_{\Omega}-k^2(\partial_r\tau\omega\hat{v},\omega\phi)_{\Omega} = 0,\label{weakfortrunhe2}
	\end{align}
	for all $\phi\in H_0^1(\Omega),$ and
	\begin{align*}
		\omega=\begin{cases}
			1,  & {\bs x}\in\Omega_1, \\
			e^{{\rm i}k(\tau(r,\theta)-a(\theta))},  & {\bs x}\in \Omega_2.
		\end{cases}
	\end{align*}
We now introduce  finite element discretisation for \eqref{weakfortrunhe2}, we need to introduce some notations first. Let  $\mathcal{T}_h$ be a regular triangulation of the domain $\Omega_1\cup \Omega_2$, and $K\in\mathcal{T}_h$ be an element. Let $P_N(K)$ be the polynomial set of degree at most $k$ on the element $K$. Define the FEM space
	$$V_h=\{\phi\in C(\bar \Omega):\phi|_K\in P_N(K)\}.$$
	Then, the FEM for \eqref{weakfortrunhe2}  is defined as: Find $\hat{u}_h = w\hat{v}_h$ with $\hat{v}_h\in V_h$ such that $\hat{v}_h = g_h$ on $\Gamma_D$, $\hat{v}_h = 0$ on $\Gamma_2$ such that
	\begin{align}
		\widehat{\mathcal{B}}(\hat{v}_h,\phi)=\mathcal{B}(\omega\hat{v}_h,\omega\phi) = \Big(\bs{C}\nabla (\omega \hat{v}_h),\nabla\big(\frac{r\omega}{\tau(r,\theta)}\phi\big)\Big)_{\Omega}-k^2\big(\partial_r\tau\omega\hat{v}_h,\omega\phi\big)_{\Omega} = 0,\label{Gamerlinscheme1}
	\end{align}
	for all $\phi\in V_h^0=H^1_0(\Omega)\cap V_h.$ Here $g_h$ denotes the $L^2$ projection of $g$. In real implementation, we derive from direct calculation that
		\begin{align}
			\widehat{\mathcal{B}}(\hat{v}_h,\phi) = (\alpha_1\nabla \hat{v}_h,\nabla\phi)_{\Omega}+(\alpha_2\cdot\nabla\hat{v}_h,\phi)_{\Omega}+(\alpha_3\hat{v}_h,\nabla\phi)_{\Omega}+(\alpha_4\hat{v}_h,\phi)_{\Omega},\label{computarcl}
		\end{align}
		where
		\begin{align}
			&\alpha_1 = 1,\;\;  \alpha_2 = 0,\; \; \alpha_3 = 0,\;\; \alpha_4 = -k^2 \;\; \hbox{ in } \;\; \Omega_1;\nonumber\\
			&\alpha_1 = \frac{r}{\tau}\bs{C}, \;\; \alpha_2 = \Big(\frac{1}{r\partial_r\tau},\; \frac{\partial_\theta\tau}{\tau^2}\Big)\bs{R}_\theta^\top-{\rm i}k\Big(1+\frac{\partial_r\tau a'(\theta)}{\tau^2},-\frac{\tau_0 r a'(\theta)}{\tau}\Big){\bs{R}}_\theta^\top,\nonumber\\
			&\alpha_3 = {\rm i}k{\bs{R}}_\theta\Big(1+\frac{\partial_\theta\tau a'(\theta)}{\tau^2},-\frac{\tau_0 r a'(\theta)}{\tau}\Big)^\top,\nonumber\\
   &\alpha_4 = k^2\frac{\tau_0 a'(\theta)}{\tau}+{\rm i}k\Big(\frac{1-\tau_0 r}{r}+\frac{\partial_\theta\tau a'(\theta)}{r\tau^2}\Big)\;\; \hbox{ in } \;\;  \Omega_2.\nonumber
		\end{align}


 \subsubsection{Accuracy tests}
 We first consider a square scatterer $D$ centered at the origin with width $0.8$ (see Figure \ref{fufig2}).  It is known from separation of variable that the Helmholtz equation in free space $\mathbb R^2$ has the exact solution: $H_0^{(1)}(k \rho)$, so we take $g=H_0^{(1)}(k a(\theta))$ on $\Gamma_D.$ As a result, we have the exact  solution to   calculate the numerical errors. For example, in the layer $\Omega_2$, we have $v =H_0^{(1)}(k \tau(r,\theta))e^{-{\rm i}k(\tau(r,\theta)-a(\theta))}.$
 Here, we set $d_1=d_2=0.3$ and choose $\tau_0$ such that $e^{-\frac{1}{2}\tau_0\min\{d_1,d_2\}}=\epsilon=10^{-12}.$
 Tables \ref{tableconv1}--\ref{tableconv2} shows the corresponding errors and convergent orders in the $L^2$-norm for piecewise linear (i.e., $N=1$) and quadratic FEM  (i.e., $N=2$) for $k=10$. Table \ref{tableconv3} tabulates the data for FEM with $N=4$ and relatively higher wave number: $k=50$.  Here, we measure the $L^2$-errors in both $\Omega_1$ and $\Omega_2,$ compared with the exact solutions $u$ and $v$ in terms of the real and imaginary parts: $e_u^R, e_u^I, e_v^R, e_v^I$, respectively.
	\begin{table}[!ht]
 \caption{Convergence rate of linear FEM ($N=1$) with $k  = 10$.}
		\label{tableconv1} \small
		\begin{tabular}{|c|cc|cc|cc|cc|}
			 \hline
			mesh &$\|e_u^{R}\|_{L^2(\Omega_1)}$ & order &$\|e_u^{I}\|_{L^2(\Omega_1)}$ & order  &$\|e_v^{R}\|_{L^2(\Omega_{2})}$ & order &$\|e_v^{I}\|_{L^2(\Omega_{2})}$ & order \\
			\hline
			$32\times 32$ & 4.2926{\rm e}-2 &     &2.0148{\rm e}-2 &    & 3.5546{\rm e}-3 & & 2.2768{\rm e}-3&  \\
			$64\times 64$ & 9.0822{\rm e}-3 &  2.2407   &6.2923{\rm e}-3 & 1.6790 & 1.5050{\rm e}-3 & 1.2399 & 7.7215{\rm e}-4& 1.5601 \\
			$128\times 128$ & 2.3938{\rm e}-3 & 1.9237 & 1.7993{\rm e}-3 & 1.8062&4.5948{\rm e}-4 &1.7117 &2.3579{\rm e}-4 & 1.7114 \\
			$256\times 256$ & 6.2114{\rm e}-4 & 1.9463 & 4.7037{\rm e}-4 & 1.9356 &1.2152{\rm e}-4 & 1.9188 &6.4456{\rm e}-5&1.8711\\
			$512\times 512$& 1.5763{\rm e}-4& 1.9784 &1.1685{\rm e}-4 & 2.0091&3.1578{\rm e}-5 &1.9442    &1.7054{\rm e}-5&1.9182  \\
			\hline
		\end{tabular}
	\end{table}
	
	\begin{table}[!th]
 		\caption{Convergence rate of quadratic FEM ($N=2$) with $k  = 10$.}
		\label{tableconv2} \small
		\begin{tabular}{|c|cc|cc|cc|cc|}
			\hline
			mesh &$\|e_u^{R}\|_{L^2(\Omega_1)}$ & order &$\|e_u^{I}\|_{L^2(\Omega_1)}$ & order  &$\|e_v^{R}\|_{L^2(\Omega_{2})}$ & order &$\|e_v^{I}\|_{L^2(\Omega_{2})}$ & order \\
			\hline
			$32\times 32$ &8.1397{\rm e}-3 &       &4.4495{\rm e}-3 &        & 9.1988{\rm e}-5 &    & 1.3017{\rm e}-4&  \\
			$64\times 64$ & 1.1463{\rm e}-3 & 2.8280 & 1.5706{\rm e}-3 &1.5023&2.3888{\rm e}-5 &1.9452 &1.2189{\rm e}-5 &3.4167\\
			$128\times 128$ & 1.5962{\rm e}-4 &2.8443&1.2899{\rm e}-4 &3.6060&3.0490{\rm e}-6 &2.9699    &1.4703{\rm e}-6&3.0514\\
			$256\times 256$ & 1.1351{\rm e}-5 &3.8138&9.5207{\rm e}-6 &3.7600&2.3101{\rm e}-7 & 3.7223   &2.5842{\rm e}-7&2.5083\\
			\hline
		\end{tabular}
		\end{table}

	\begin{table}[!ht]
		\caption{Convergence rate of FEM ($N=4$) with $k  = 50$.} \small
		\label{tableconv3}
		\begin{tabular}{|c|cc|cc|cc|cc|}
			\hline
			mesh &$\|e_u^{R}\|_{L^2(\Omega_1)}$ & order &$\|e_u^{I}\|_{L^2(\Omega_1)}$ & order  &$\|e_v^{R}\|_{L^2(\Omega_{2})}$ & order &$\|e_v^{I}\|_{L^2(\Omega_{2})}$ & order \\
			\hline
			$32\times 32$ &1.2654{\rm e}-2 &       &1.3659{\rm e}-2 &        & 1.3070{\rm e}-3 &    & 1.4873{\rm e}-3&  \\
			$64\times 64$ & 6.7773{\rm e}-4 & 4.2227 & 6.4418{\rm e}-4 &4.4062&8.2377{\rm e}-5 &3.9878 &8.6130{\rm e}-5 &4.1100\\
			$128\times 128$ & 1.7829{\rm e}-5 &5.2484&1.8616{\rm e}-5 &5.1128&9.5980{\rm e}-7 &6.4233    &1.5095{\rm e}-6&5.8344\\
			$256\times 256$ & 6.0015{\rm e}-7 &4.8927&6.2044{\rm e}-7 &4.9070&4.1087{\rm e}-8 & 4.5449   &6.5804{\rm e}-8&4.5197\\
			\hline
		\end{tabular}
	\end{table}

 In Figures  \ref{fufig2} and \ref{fufig3}, we depict the 2D plots of  wave propagation and profiles along the $x_1$-axis.
 Some observations from tables and figures are listed in order.
 \begin{itemize}
 \item The orders of convergence are as expected for typical FEM approximations. The use of higher order elements is necessary for higher wave numbers.
 \smallskip
 \item Thanks to the substitution in $\Omega_2,$ the approximation of $v$  is more accurate in magnitude of the errors.
 \smallskip
 \item In the layer $\Omega_2$, the fields are essentially free of oscillations. Moreover, the coefficients of the RCL-equation are all real. These show the robustness and non-reflective nature of this new technique.
\end{itemize}


	\begin{figure}[!ht]
		\centering
		\subfigure[Real part of numerical solution]{
			\includegraphics[width=0.40\textwidth]{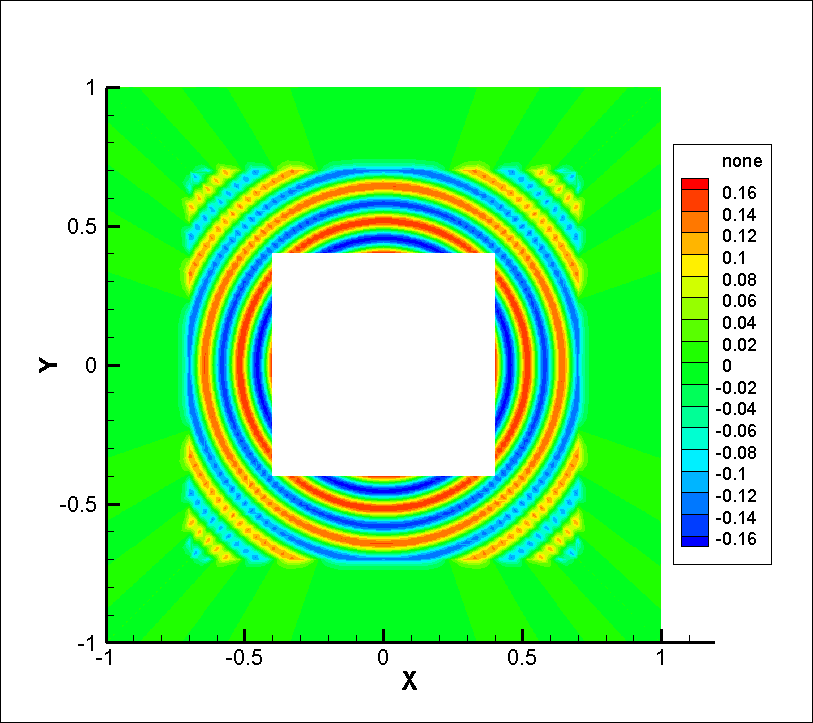}
		}
		\subfigure[Imaginary part numerical solution]{
			\includegraphics[width=0.40\textwidth]{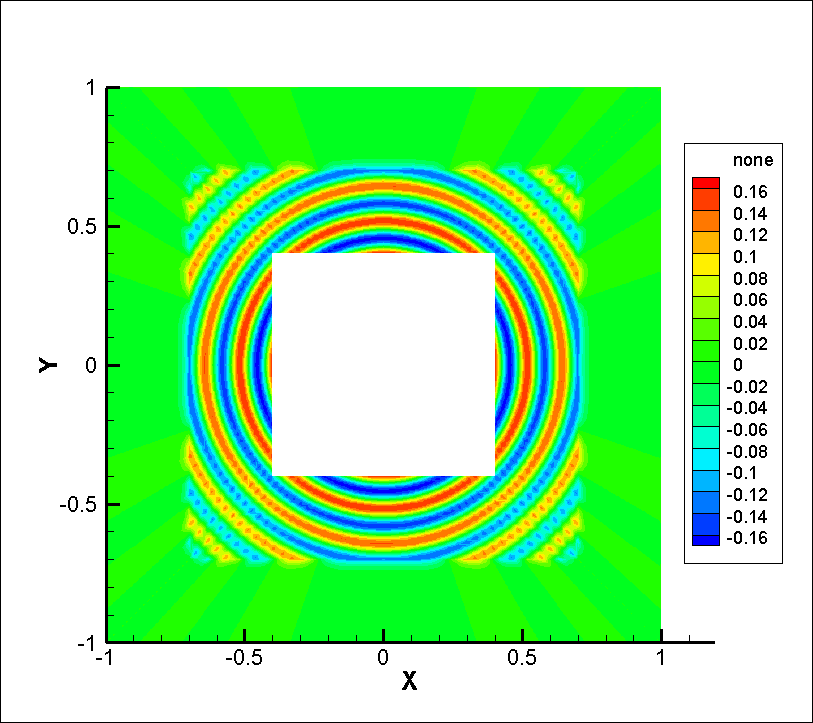}
		}
		\subfigure[Real part of exact solution]{
			\includegraphics[width=0.40\textwidth]{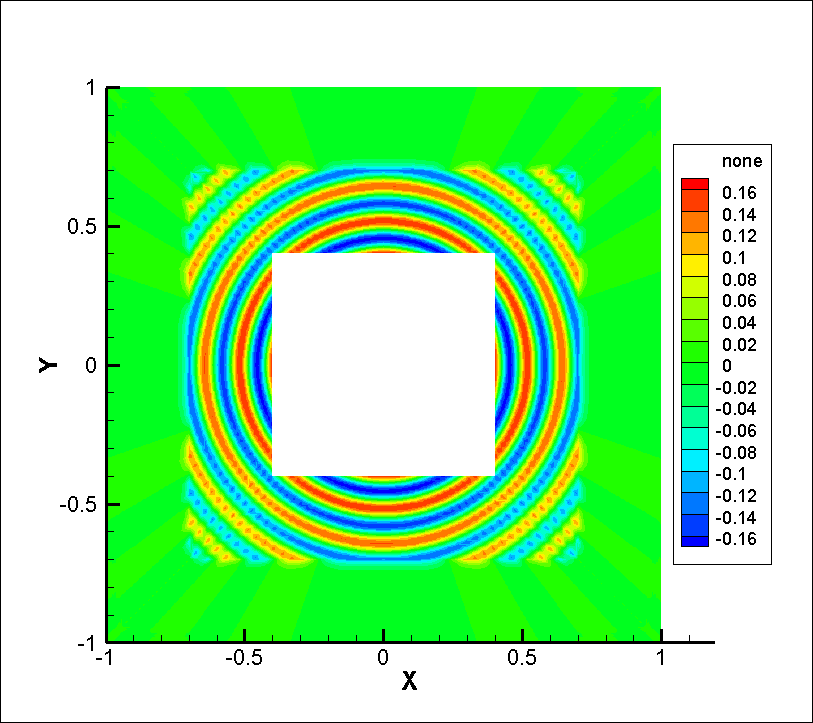}
		}
		\subfigure[Imaginary part of exact solution]{
			\includegraphics[width=0.40\textwidth]{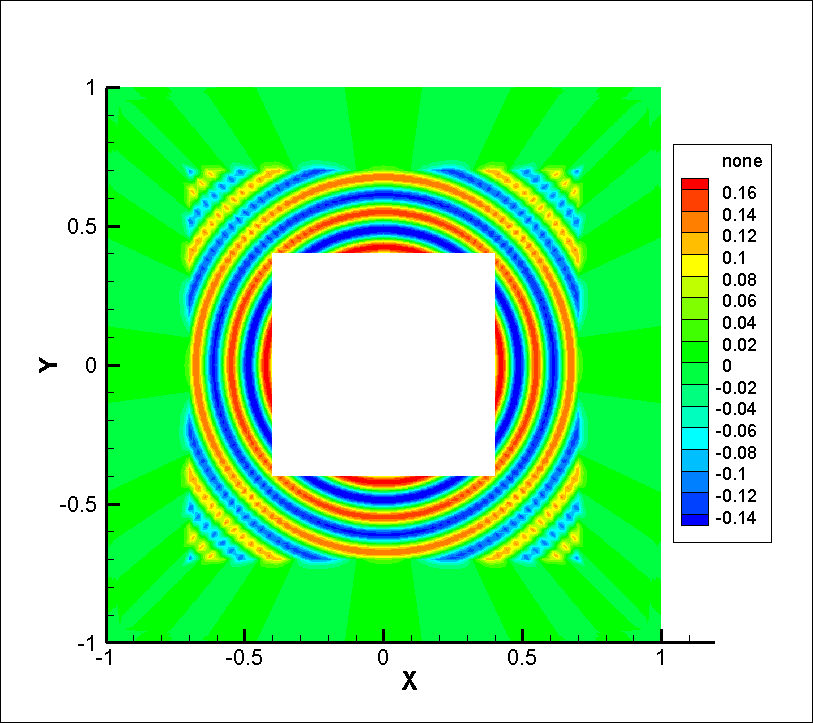}
		}
		\caption{2D plots of numerical and exact solutions  with $k  = 50, N = 4.$}\label{fufig2}
	\end{figure}
	\begin{figure}[!th]
		\centering
		\subfigure[Real part]{
			\includegraphics[width=0.40\textwidth]{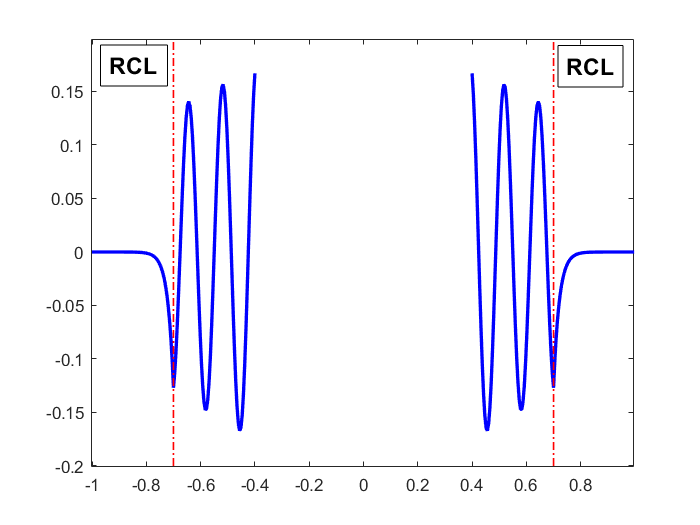}
		}
		\subfigure[Imaginary part]{
			\includegraphics[width=0.40\textwidth]{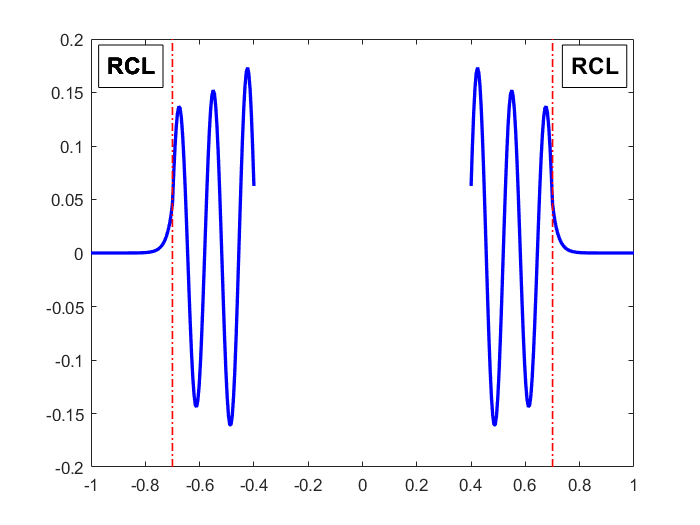}
		}
		\caption{Profiles of the numerical solution along $x_1$-axis  with $k = 50, N= 4$ in 2D.}\label{fufig3}
	\end{figure}

\subsubsection{Application to an $L$-shaped scatterer}
 As an application, we take the scatterer $D$ to be an $L$-shaped domain contained in the square, and use a setting as in the tests above. Again we see that the fields in $\Omega_2$ have essentially no oscillations.

	\begin{figure}[!ht]
		\centering
  \subfigure[Real part]{
			\includegraphics[width=0.40\textwidth]{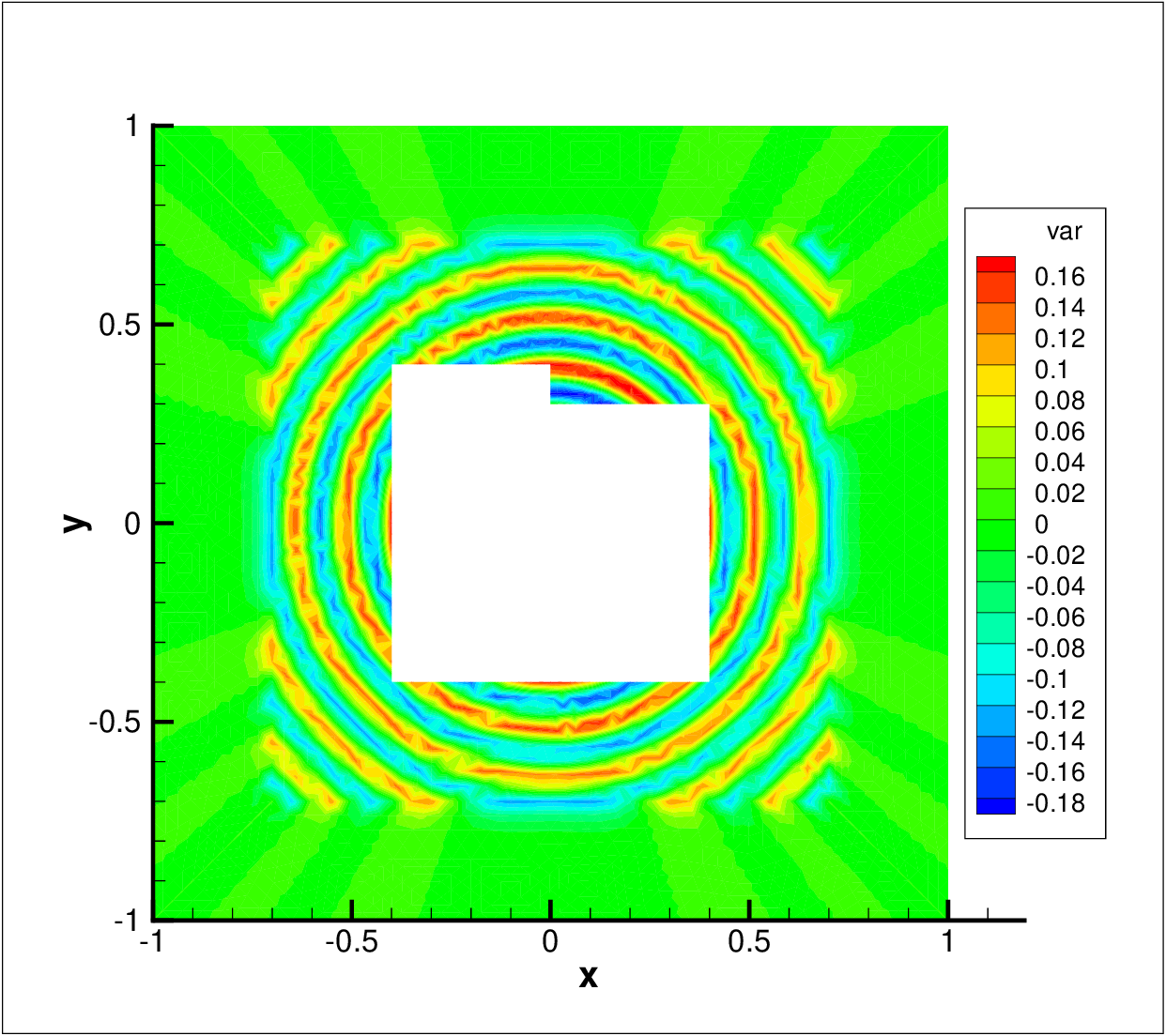}
		}\qquad
		\subfigure[Imaginary part]{
			\includegraphics[width=0.40\textwidth]{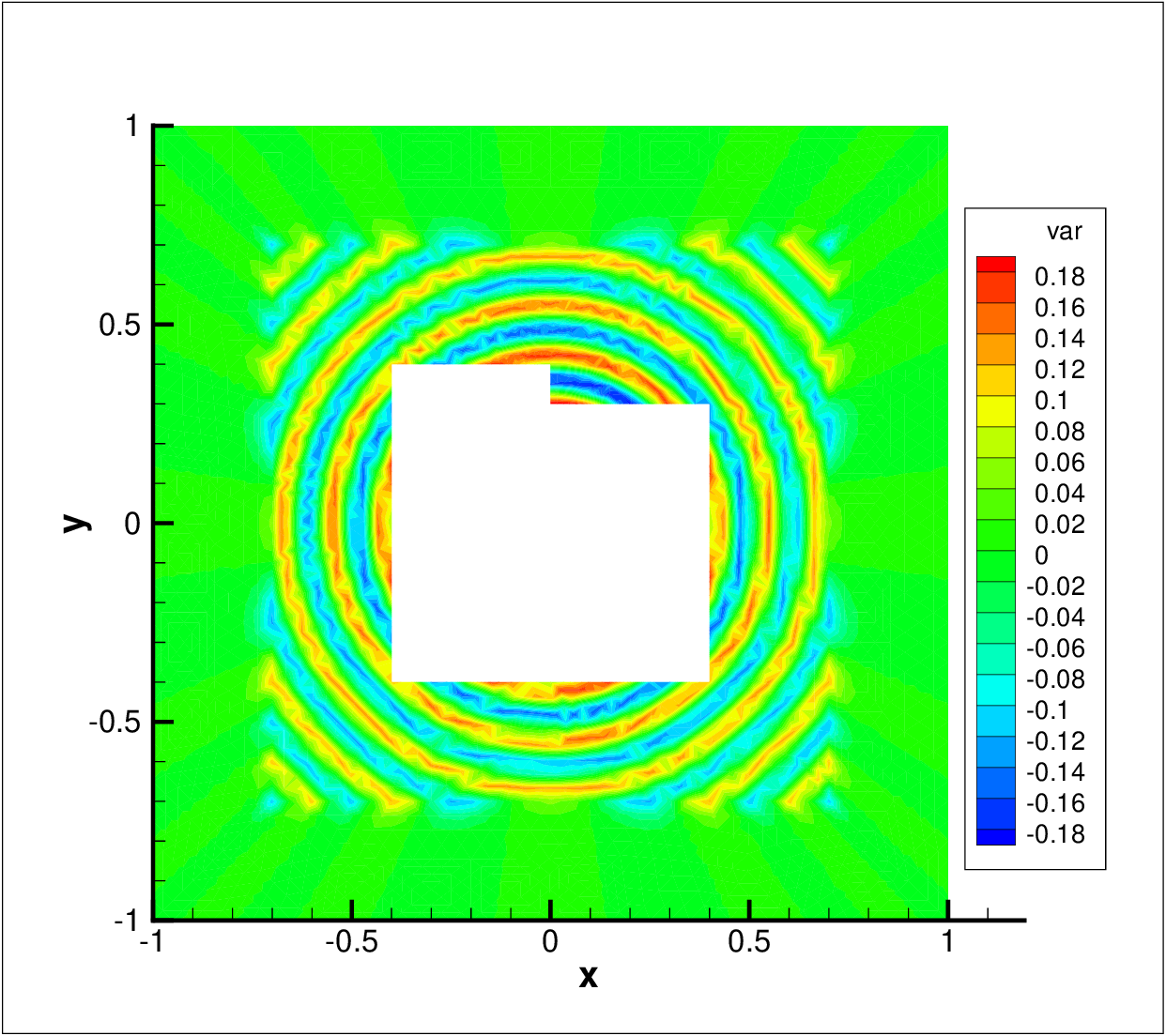}
		}
		\caption{Waves scattered by an $L$-shaped scatterer simulated by the FEM scheme with $N=4$ and $k=50$. }\label{fufig4}
	\end{figure}

	
	\section{Concluding remarks}

	In summary, we proposed a PML-type technique for domain reduction of time-harmonic acoustic wave scattering problems based on a suitable real coordinate transformation.
 Although it has been long known  that the naive use of RCT is not feasible for scattering waves, we showed that (i) the properly chosen RCT could induce an exponential decay factor from slow decay factor in the original coordinates; and  (ii) the resulting highly oscillatory factor could be explicitly extracted and built in the numerical solver. We demonstrated that this new technique is robust and non-reflective for high wave numbers. Compared with all existing techniques based on complex coordinate stretching, the PDE in the layer has real coefficients and the computed fields  can provide a reasonable recovery of the far-field outgoing waves.


 Here, we presented this novel technique in two dimensions, but the ideas can be extended to three dimensions. It is certainly of interest to explore the time-domain RCL and other type of wave propagation. Moreover, we have noticed the recent works on different perspectives of the Helmholtz problems including PML and related theoretical aspects, see e.g., \cite{chaumont2022wavenumber,galkowski2022hp,jiang2022finite,li2023new,li2019fem,li2020cip,zhu2013preasymptotic}.
	
	\bibliographystyle{siam}
	\bibliography{RPML}
\end{document}